\documentclass{article}

\usepackage[english]{babel}

\usepackage[letterpaper,top=2cm,bottom=2cm,left=3cm,right=3cm,marginparwidth=1.75cm]{geometry}

\usepackage{amsmath,amssymb,amsfonts,amsthm}
\usepackage{graphicx,enumerate,array,xcolor}
\usepackage[colorlinks=true, allcolors=black]{hyperref}
\usepackage{algorithmic,algorithm,multicol}
\usepackage{makecell}

\usepackage{url, hyperref}
\usepackage{pgf}
\usepackage{multirow}
\usepackage{subfig}
\usepackage{booktabs}
\usepackage{array}
\usepackage{amsmath}

\usepackage{tikz}
\usetikzlibrary{positioning}
\usetikzlibrary{decorations.pathreplacing}

\newtheorem{assumption}{Assumption}

\newtheorem{theorem}{Theorem}
\newtheorem{proposition}{Proposition}

\newtheorem*{theorem*}{Theorem}
\newtheorem{lemma}{Lemma}
\newtheorem{remark}{Remark}

\newcommand\keywords[1]{\textbf{Keywords}: #1}

\newcommand{\CL}{{\mathcal{L}}}

\newcommand{\CG}{{\mathcal{G}}}
\newcommand{\CO}{{\mathcal{O}}}

\newcommand{\CS}{{\mathcal{S}}}

\newcommand{\BR}{{\mathbb{R}}}

\newcommand{\la}{{\langle}}
\newcommand{\ra}{{\rangle}}

\newtheoremstyle{subproofstyle}%
{}{}%
{\itshape}{}%
{\bfseries}{}%
{\newline}%
{\thmname{#1} \thmnumber{#2} \thmnote{ #3}}

\theoremstyle{subproofstyle}

\usepackage{pifont}

\title{A Simple Adaptive Proximal Gradient Method for Nonconvex Optimization \thanks{\textbf{Funding: }{This work was supported by the National Science Foundation Division of Mathematical Sciences 
[Grant 2243650]; the Division of Computing and Communication Foundations [Grants 2308597 and 
2311275]; the Division of Electrical, Communications and Cyber Systems [Grant 2326591]; the National 
Natural Science Foundation of China [Grants 12431011 and 12371301]; the Key Laboratory 
of NSLSCS, Ministry of Education of China.}}}

\author{Zilong Ye$^{\dagger}$, 
Shiqian Ma$^{\ddagger}$,
Junfeng Yang$^{\dagger}$
and Danqing Zhou$^{\S}$}

\begin{document}

\maketitle

\footnotetext[2]{School of Mathematics, Nanjing University, Nanjing, Jiangsu 210093, People's Republic of China (Email: \texttt{zilongye@smail.nju.edu.cn, jfyang@nju.edu.cn}).}
\footnotetext[3]{Department of Computational Applied Math and Operations Research, Rice University, Houston, Texas 77005 (Email: \texttt{sqma@rice.edu})}
\footnotetext[4]{School of Mathematics, Nanjing Audit University, Nanjing, Jiangsu 211815, People's Republic of China (Email: \texttt{dqkyo@foxmail.com})}

\begin{abstract}
Consider composite nonconvex optimization problems where the objective function consists of a smooth nonconvex term (with Lipschitz-continuous gradient) and a convex (possibly nonsmooth) term.
Existing parameter-free methods for such problems often rely on complex multi-loop structures, require line searches, or depend on restrictive assumptions (e.g., bounded iterates). To address these limitations, we introduce a novel adaptive proximal gradient method (referred to as AdaPGNC) that features a simple single-loop structure, eliminates the need for line searches, and only requires the gradient’s Lipschitz continuity to ensure convergence. Furthermore, AdaPGNC achieves the theoretically optimal iteration/gradient evaluation complexity of $\CO(\varepsilon^{-2})$ for finding an $\varepsilon$-stationary point. Our core innovation lies in designing an adaptive step size strategy that leverages upper and lower curvature estimates. A key technical contribution is the development of a novel Lyapunov function that effectively balances the function value gap and the norm-squared of consecutive iterate differences, serving as a central component in our convergence analysis.  Preliminary experimental results indicate that AdaPGNC demonstrates competitive performance on several benchmark nonconvex (and convex) problems against state-of-the-art parameter-free methods.
\end{abstract}


\section{Introduction}\label{section-introduction}

We consider the composite optimization problem:
\begin{align}\label{problem}
    \min_{x\in\BR^n} F(x): = f(x)+h(x), 
\end{align}
where $f:\BR^n\rightarrow \BR$ is continuously differentiable and $h:\BR^n \rightarrow \BR$ is a proper, closed and convex function (possibly nonsmooth). Throughout the paper, we assume that the set of solutions of problem \eqref{problem} is nonempty, i.e., there exists $x_*\in\BR^n$ such that $F(x_*)=F_*\triangleq\inf_{x\in\BR^n} F(x)>-\infty$. A classical method for solving this problem is the proximal gradient method, which follows the following iteration formula:
\begin{align}\label{pgd}
    x_{k+1} = \text{prox}_{\lambda_k h}(x_k - \lambda_k \nabla f(x_k)), \quad k = 0, 1, 2, \ldots.
\end{align}
Here, $x_0\in\BR^n$ is an arbitrary starting point, and $\lambda_k >0$ represents the step size at the $k$-th iteration. The proximal operator of $h$ is defined as
\begin{align*}
    \text{prox}_{th}(x) := \arg\min_{y\in\BR^n} \Big\{ h(y) + \frac{1}{2t} \|y - x\|^2 \Big\}, \quad \forall x\in\BR^n, \, t>0. 
\end{align*}
Note that \eqref{pgd} reduces to the gradient descent (GD) if $h\equiv 0$. The key question to ensure convergence lies in how to determine the step size $\lambda_k$ at each iteration. When $f$ is convex and globally $L$-smooth, its gradient operator $\nabla f$ satisfies
\begin{align}\label{convex-ieq}
    f(x) + \langle \nabla f(x), y-x\rangle \leq f(y), \quad \forall x, y \in \BR^n, 
\end{align}
and there exists a constant $L>0$ such that 
\begin{align}\label{lipschitz-smooth}
    \|\nabla f(x)-\nabla f(y)\|\leq L \| x-y \|, \quad \forall x, y \in \BR^n.
\end{align}
Under such conditions, any constant step size $\lambda_k = \lambda \in (0,2/L)$ suffices to ensure the pointwise convergence of the sequence $\{x_k\}$ generated by \eqref{pgd} to a minimizer of $f$; see, e.g., \cite[Thm. 10.24]{beck2017first}. Moreover, in this setting, the function value gap $F(x_k) - F_*$ converges at a sublinear rate of $\CO(1/k)$; see, e.g., \cite[Thm. 10.21]{beck2017first}. Nevertheless, determining the Lipschitz constant is challenging in practical settings. Consequently, $\lambda_k$ needs to be tuned repeatedly, or we can determine $\lambda_k$ via line-search, such as backtracking or Wolfe-Powell line-search. However, line-search entails additional computational overhead: for each trial point, the evaluation of the function value and, in some cases, the gradient is required. This can be impractical in certain applications. For instance, in the context of large-scale problems in machine learning, computing a single function value typically necessitates a full pass over the dataset, an operation that is often prohibitive. This motivates the adoption of simpler step size selection strategies, such as diminishing step sizes or adaptive step size determination that adjusts dynamically throughout the iteration process. 

Recently, extensive research has focused on adaptive step size selection throughout the iteration process. The key motivations for this line of work are as follows: First, computing the global Lipschitz constant $L$ can be computationally expensive. For instance, in logistic regression, $L$ is proportional to the maximum eigenvalue of $X^\top X$ (where $X$ denotes the data matrix). Calculating this maximum eigenvalue is challenging in high-dimensional settings, and becomes even more cumbersome in federated or distributed frameworks, where data are stored across distinct locations. Second, gradient descent with fixed step sizes derived from the worst-case $L$ tends to be overly conservative in practice. This leads to inefficient updates in ``flat" regions of the objective function, where significantly larger step sizes would be permissible. Finally, the objective function $f$ may only be locally $L$-smooth, that is, the inequality in \eqref{lipschitz-smooth} holds solely for $x, y$ within a bounded region. A typical example of this scenario is convex optimization problems with cubic regularization. A seminal contribution to this line of research is due to Malitsky and Mishchenko \cite{malitsky2019adaptive}, who proposed an adaptive gradient descent (AdGD) method for \eqref{problem} with $h\equiv0$, utilizing the following step size rule:
\begin{equation}\label{malalgostep-1}
\lambda_ k = \min\left\{\sqrt{1+\theta_{k-1}}\lambda_{k-1}, \frac{1}{2L_k}\right\}
\text{~~with~~}L_k = \frac{\|\nabla f(x_k)-\nabla f(x_{k-1})\|}{\|x_k-x_{k-1}\|},
\quad \theta_{k-1} = \frac{\lambda_{k-1}}{\lambda_{k-2}}. 
\end{equation}
Here, $L_k$ estimates the local curvature of $f$, while the first term in the minimum defining  $\lambda_k$ regulates the growth rate of the step size. This step size strategy guarantees the monotonic decrease of a specific Lyapunov function, which in turn ensures the boundedness of the iterates and the convergence of AdGD under convexity and the weaker local $L$-smoothness condition. Subsequently, AdGD was refined in \cite{malitsky2024adaptive} and extended to composite convex optimization problems in \cite{latafat2024adaptivefirstorder,latafat2024adaptiveproximal}. More recently, Zhou et al. \cite{zhou2024adabb} proposed an adaptive gradient method based on the short Barzilai-Borwein step size \cite{barzilai1988two} for (composite) convex optimization problems, which achieves the same convergence results as those in the globally $L$-smooth setting. We also mention the recent work \cite{li2025simple} and \cite{suh2025adaptiveparameterfreenesterovsaccelerated}, which use an adaptive estimation for the Lipschitz constant and achieve the accelerated convergence rate in function value gap: $F(x_k) - F_* = \mathcal{O}(1/k^2)$.

In this paper, we further investigate the nonconvex yet globally $L$-smooth setting, where the Lipschitz smoothness condition \eqref{lipschitz-smooth} holds for some constant $L>0$ and all $x, y \in \mathbb{R}^n$. Under the $L$-smoothness condition \eqref{lipschitz-smooth}, there exists a constant $l\in [0,L]$ such that 
\begin{align}\label{curvatures}
    -\frac{l}{2}\|x-y\|^2 \leq f(x) - f(y) - \langle \nabla f(y), x-y\rangle \leq \frac{L}{2}\|x-y\|^2, \quad \forall x,y\in\BR^n.
\end{align}
When $l=0$, the function $f$ is convex; otherwise, $f$ is weakly convex, following the terminology in, e.g., \cite{davis2019stochastic,drusvyatskiy2017proximal}. In the following, we refer to $L$ and $l$ as the upper curvature and lower curvature, respectively. We propose an algorithm termed AdaPGNC, which requires no prior knowledge of $L$ and $l$, a property that renders the algorithm parameter-free. We shall estimate $L$ and $l$ using the current and previous iterates, and the resulting estimates are then used to construct the step sizes.  Furthermore, a key technique underpinning our algorithm is the introduction of a summable nonnegative sequence to control the growth rate of step sizes, a strategy that has also been adopted in recent work \cite{hoaia2025composite, hoai2024novel, liu2022some}. 

In the rest of this section, we first review several existing parameter-free methods for nonconvex optimization. All these methods assume that $f$ is globally $L$-smooth and satisfies the curvature condition \eqref{curvatures} for some $l\in [0,L]$.  Recall that an algorithm is considered parameter-free if it does not require prior knowledge of the upper and lower curvature parameters $L$ and $l$. We then summarize the contributions of this paper and its organization.

\subsection{Existing parameter-free methods for nonconvex optimization} 
Recently, Lan et al. proposed a parameter-free method named NASCAR (\emph{Nonconvex Acceleration via Strongly Convex Accumulative Regularization}, \cite[Alg. 7.2]{lan2024optimalparameterfreegradientminimization}) for solving the nonconvex optimization problem \eqref{problem} with $h\equiv0$. It requires solving a sequence of strongly convex regularization subproblems, involves four nested loops, and relies on line-search. Furthermore, it needs $\mathcal{O}(1/\varepsilon^2)$ gradient evaluations to reach an $\varepsilon$-stationary point $x$ satisfying $\|\nabla f(x)\|\leq \varepsilon$. While this result is optimal for $l = L$ \cite{carmon2020lower} and represents the state of the art for $l \in (0, L)$, the four-nested-loop structure of the method undermines its practical effectiveness. Another limitation of NASCAR is that it requires the precision parameter $\varepsilon$ as an input to the algorithm, since $\varepsilon$ is necessary for solving its sequence of strongly convex regularization subproblems. Marumo and Takeda proposed a parameter-free method named AGDR (\emph{Accelerated Gradient Descent with Restart}, \cite[Alg. 4.1]{marumo2024parameterfree}) under the same setting as Lan et al. \cite[Alg. 7.2]{lan2024optimalparameterfreegradientminimization}, but with an extra second-order smoothness condition. AGDR features a simple single-loop structure but relies on a restart technique, and it requires ${\cal O}(1/\varepsilon^{7/4})$ gradient evaluations to reach an $\varepsilon$-stationary point. Notably, this improved convergence result is attributed to the aforementioned additional second-order smoothness assumption. However, these two algorithms cannot solve the general composite problem \eqref{problem}. Kong proposed a parameter-free method named PF.APD (\emph{Parameter-Free with Accelerated Proximal Descent}, \cite[Alg. 3.4]{kong2024complexityoptimalparameterfreefirstordermethods}) to solve the nonconvex composite optimization problem \eqref{problem}. The algorithmic structure and theoretical results of PF.APD are completely analogous to those of NASCAR \cite[Alg. 7.2]{lan2024optimalparameterfreegradientminimization}: it iteratively solves a sequence of strongly convex regularization subproblems, involves four nested loops, relies on line-search, and requires $\mathcal{O}(1/\varepsilon^2)$ gradient evaluations to reach an $\varepsilon$-stationary point $x$ satisfying $\|\nabla f(x) + h'(x)\| \leq \varepsilon$ for some subgradient $h'(x)\in\partial h(x)$. For problem \eqref{problem} where $h$ being the indicator function of a compact convex set $\Omega$ (i.e., a nonconvex optimization problem with a compact convex constraint), Lan et al. proposed a parameter-free method called AC-PG (\emph{Auto-Conditioned Projected Gradient}, \cite[Alg. 2]{lan2024projected}). AC-PG method adopts a single loop structure yet employs a monotonically nonincreasing step size strategy based on the local curvature estimation:
\begin{align*}
    \CL_k =  2\big(f(x_k)-f(x_{k-1})-\langle \nabla f(x_{k-1}),x_k-x_{k-1}\rangle\big) \big/ 
    \|x_k-x_{k-1}\|^2.
\end{align*}
Specifically, AC-PG determines the step size using historical information of $\CL_k$:  $$\lambda_{k+1} = 1/\max\{\CL_0,\CL_1,\dots, \CL_k\}>0\; \forall k\geq0 \text{ and }\lambda_0= 1/\CL_0 \geq 1/L.$$ To ensure the projected gradient has a norm no greater than $\varepsilon$, the AC-PG method requires ${\cal O}(1/\varepsilon^2)$ iterations. This complexity result heavily depends on the boundedness of $\Omega$; in particular, the squared diameter of $\Omega$ is concealed within the big ${\cal O}$ notation. Very recently, Hoai and Thai proposed a NPG2 method (\emph{nonconvex proximal gradient}, \cite[Alg. 4.2]{hoaia2025composite}) for solving the composite nonconvex problem \eqref{problem}, where $f$ is assumed to satisfy that for any $u,v$, the function $\langle \nabla f(u+t(v-u)), v-u\rangle$ is quasiconvex with respect to $t\in [0,1]$.
This method estimates the upper curvature $L$ as in \eqref{malalgostep-1} and adopts a summable nonnegative sequence technique to control the growth rate of step sizes.
Furthermore, NPG2 features a simple single-loop structure and achieves an ${\cal O}(1/\varepsilon^2)$ gradient evaluation complexity, ensuring the proximal gradient has a norm not greater than $\varepsilon$ (see \cite[Thm. 4.1]{hoaia2025composite} for details).

\par While the aforementioned algorithms achieve parameter-free adaptation, each is plagued by notable limitations. Specifically, both NASCAR and PF.APD feature four nested loops and necessitate line-search for estimating the local upper curvature, which leads to prohibitively complex implementations and renders their efficiency far from practical. Although AC-PG simplifies the algorithmic structure, it is restricted to problems with compact convex constraints. On the other hand, NPG2 relaxes this boundedness requirement but imposes a quasiconvexity assumption on $\langle \nabla f(u+t(v-u)), v-u\rangle$ with respect to $t\in [0,1]$, a condition that may prove challenging to verify in practice. Regarding AGDR, while it can solve \eqref{problem}, it requires second-order smoothness for conducting its convergence analysis, greatly limiting its applicability.

\subsection{Summary of contributions}

To the best of our knowledge, no existing method simultaneously satisfies all the following features: (a) having a simple single-loop structure; (b) admitting convergence analysis under first-order Lipschitz continuity alone (with no additional assumptions); (c) requiring no line-search or restart; (d) being applicable to the general composite nonconvex problem \eqref{problem}. To bridge this gap, we propose AdaPGNC (\emph{Adaptive Proximal Gradient for NonConvex Optimization}) in this paper, which achieves all the aforementioned features simultaneously. Specifically, AdaPGNC is parameter-free and line-search-free, adopts a simple single-loop structure, and requires no restart. Additionally, it solves the generic composite nonconvex problem \eqref{problem}. Furthermore, under the blanket global $L$-smooth condition \eqref{lipschitz-smooth}, we show that AdaPGNC finds an $\varepsilon$-stationary point within $\CO(1/\varepsilon^{2})$ iterations, which equals the number of gradient evaluations, as the algorithm only requires one gradient computation per iteration. This complexity result matches the state-of-the-art for $l\in(0,L)$ and is optimal when $l=L$. At each iteration, AdaPGNC adaptively determines the step size by incorporating local curvature information and restricts step size growth via a nonnegative summable sequence. It naturally extends Malitsky and Mishchenko's AdGD algorithm (see \cite{malitsky2019adaptive}) to general nonconvex problems with a convex regularizer. For a clear comparison, the features of the reviewed parameter-free methods and our AdaPGNC are summarized in Table \ref{tablesummarypfnonconvexmethod}.

\begin{table}[ht]
    \centering
    \renewcommand\arraystretch{1.5}
    \caption{Comparison of features, restrictions, and/or additional assumptions (excluding the blanket global $L$-smooth condition) for the reviewed parameter-free methods and our AdaPGNC. Here, ``LS" denotes line-search.} 
    \begin{tabular}{|c|c|c|c|c|}  
    \hline
        \multicolumn{1}{|c|}{Algorithm} & \makecell[c]{Restrictions \\ \&/or Additional Assump.} & \makecell[c]{ Without \\ LS } & \makecell[c]{ Without \\ Restart } &\makecell[c]{\#Loops} \\
        \hline
        \makecell[l]{NASCAR \cite[Alg. 7.2]{lan2024optimalparameterfreegradientminimization}} & $h\equiv 0$ & \ding{55} & \ding{51} & 4 \\
        \hline
        \makecell[l]{AGDR \cite[Alg. 4.1]{marumo2024parameterfree}} & \makecell[c]{$h\equiv 0$ \\ 2nd-order smoothness} & \ding{51} & \ding{55} & 1  \\
        \hline
        \makecell[l]{PF.APD \cite[Alg. 3.4]{kong2024complexityoptimalparameterfreefirstordermethods}} & none & \ding{55} & \ding{51} & 4 \\
        \hline
        \makecell[l]{AC-PG \cite[Alg. 2]{lan2024projected}} & \makecell[c]{$h$ is an indicator func. \\ of a compact set $\Omega$} & \ding{51} & \ding{51} & 1 \\
        \hline
        \makecell[l]{NPG2 \cite[Alg. 4.2]{hoaia2025composite}} & \makecell[c]{$\langle \nabla f(u+t(v-u)), v-u\rangle$ \\ quasi-conv in $t$} & \ding{51} & \ding{51} & 1 \\
        \hline
        \bf{AdaPGNC (Ours)} &  \bf{none} & \bf{\ding{51}} & \bf{\ding{51}} & \bf{1} \\
        \hline
    \end{tabular}
    \label{tablesummarypfnonconvexmethod}
\end{table}

\subsection{Notation and organization}

Our notation is quite standard. Let $\BR^n$ be the $n$-dimensional Euclidean space. 
For $u,v\in\BR^n$, $\la u,v\ra$ denotes their inner product, and $\|u\|$ denotes the Euclidean norm. For a matrix $U$, we let $\|U\|_F$ be the Frobenius norm of $U$. For a proper closed convex function $h: \BR^n \rightarrow (-\infty,\infty]$, its subdifferential at a point $x\in\BR^n$ is defined as $\partial h(x)=\{v\in\BR^n~\mid~h(y)\geq h(x)+\langle v, y-x\rangle \text{~~for all~~} y \in\BR^n\}$. We adopt the convention that  $c/0 = +\infty$ for any $c>0$, and use the notation $[u]_+ = \max\{u,0\}$.

The rest of this paper is organized as follows. In Section \ref{section-algorithm}, we first present the AdaPGNC algorithm for the composite problem \eqref{problem}. Next, Section \ref{section-analysis} is dedicated to establishing the convergence results of AdaPGNC, where we rigorously derive its theoretical guarantees. Section \ref{sectionnumericalexperiments} then reports preliminary numerical results: we compare AdaPGNC with several state-of-the-art parameter-free methods on a set of benchmark nonconvex problems and the logistic regression problem to validate its practical performance. Finally, concluding remarks and potential future directions are drawn in Section \ref{sectionconclusion}.

\section{The proposed algorithm: AdaPGNC}\label{section-algorithm}

In this section, we present our AdaPGNC algorithm for the composite nonconvex optimization problem \eqref{problem}. We first formalize the blanket assumptions underlying our analysis.
\begin{assumption}\label{assumption-composite}
We assume that (i) $h:\BR^n \rightarrow(-\infty, \infty]$ is proper, closed and convex, 
(ii) $f:\BR^n \rightarrow \BR$ is globally $L$-smooth, i.e., it satisfies \eqref{lipschitz-smooth} for some $L>0$ and all $x, y\in\BR^n$, and (iii) $F$ is lower bounded, i.e., $F_*\triangleq\inf_{x\in\BR^n} F(x)>-\infty$. 
\end{assumption}

Recall that, under Assumption \ref{assumption-composite} $(ii)$, there exists some $l\in [0,L]$ such that the two inequalities in \eqref{curvatures} hold for all $x, y\in\BR^n$. 
For the smooth function $f$ in \eqref{problem}, we adopt the vanilla proximal gradient to handle its optimization.  We first estimate the upper and lower curvature of $f$ appropriately using local information at each iteration; these estimates are then used to construct the step size. Furthermore, inspired by \cite{hoaia2025composite}, we introduce a nonnegative summable sequence to regulate the step size's growth rate, which plays a critical role in the convergence analysis.  Specifically, the upper and lower curvatures of $f$ at the $k$-th iteration are estimated respectively by
\begin{align}\label{hestimate}
    L_k = \frac{\|\nabla f(x_k)-\nabla f(x_{k-1})\|}{\|x_k-x_{k-1}\|} \text{~~and~~}
    l_k = \frac{2\big(f(x_{k}) - f(x_{k-1}) +\langle \nabla f(x_{k}), x_{k-1}-x_{k}\rangle\big)}{\|x_{k}-x_{k-1}\|^2}.
\end{align}
It follows from \eqref{lipschitz-smooth} and \eqref{curvatures} that $L_k \leq L$ and $l_k \leq l \leq L$. On the other hand, for the potentially nonsmooth function $h$ in \eqref{problem}, which is proper, closed, and convex according to Assumption \ref{assumption-composite} $(i)$, we resort to the proximal operator for its optimization. The proposed AdaPGNC for solving problem \eqref{problem} is summarized in Algorithm \ref{adagdnc}. 

\begin{algorithm}[H]
    \caption{\textbf{Ada}ptive \textbf{P}roximal \textbf{G}radient Method for \textbf{N}on\textbf{C}onvex (AdaPGNC) Problem \eqref{problem}}
    \begin{algorithmic}[1]
        \STATE \textbf{Input:} $x_{0}\in\BR^n$, $\lambda_0>0$, and a nonnegative sequence 
        $\{\rho_k\}_{k=0}^\infty$ such that $\sum_{k=0}^{\infty} \rho_k <+\infty$. \smallskip 
        \STATE $x_1=\text{prox}_{\lambda_0 h}(x_0-\lambda_0 \nabla f(x_0))$ \smallskip
        \FOR{$k = 1, 2, \dots $}
        \STATE Compute $L_k$ and $l_k$ according to \eqref{hestimate}
        \IF{$l_k \leq 0$}
        \STATE $\lambda_k = \min\left\{\sqrt{1+\rho_{k-1}}\lambda_{k-1}, \frac{1}{L_k}\right\}$
        \ELSE
        \STATE $\lambda_k = \min\left\{\sqrt{1+\rho_{k-1}}\lambda_{k-1}, \frac{1}{\sqrt{2} L_k}, \sqrt{\frac{\lambda_{k-1}}{2 l_k}}\right\}$
        \ENDIF
        \STATE $x_{k+1} = \text{prox}_{\lambda_k h}(x_k - \lambda_k \nabla f(x_k))$
        \ENDFOR
    \end{algorithmic}\label{adagdnc}
\end{algorithm}

Note that from \eqref{hestimate}, the validity of $l_k\leq 0$ is equivalent to the satisfaction of the subgradient inequality \eqref{convex-ieq}, e.g., 
$f(x_{k-1})  \geq f(x_{k}) +\langle \nabla f(x_{k}), x_{k-1}-x_{k}\rangle$, which indicates the detection of local convexity.  In this case, we set $\lambda_k = \min\{\sqrt{1+\rho_{k-1}}\lambda_{k-1}, 1/L_k\}$: the first term in the minimum regulates the step size's growth rate, while the second term approximates the inverse of the upper curvature. When $l_k\leq 0$ is violated (indicating nonconvexity is detected), we adopt a more conservative step size, as specified in Line 8 of Algorithm \ref{adagdnc}. Note that the second and third terms in the minimum in Line 8 of Algorithm \ref{adagdnc} are derived from our analysis, and their setting allows some flexibility, which will be elaborated on in Remark \ref{rmk-general-ss}. \\

\textbf{Comparison with AdGD. }Our step size design for AdaPGNC builds upon the AdGD framework \cite{malitsky2019adaptive} but introduces several key modifications. First, we replace the parameter $\theta_k$ in \eqref{malalgostep-1} with $\rho_k$ to better regulate the step size growth. This change is crucial because we constrain ${\rho_k}$ to be a summable series, which prevents the potential divergence of ${\lambda_k}$ that could occur if $\theta_k$ were non-summable. Second, we employ a more aggressive step size of $1/L_k$ instead of $1/(2L_k)$ used in \eqref{malalgostep-1} for convex problems.  Third, to handle nonconvexity, we refine the step size by reducing the coefficient from $2$ to $\sqrt{2}$ and incorporating an extra term $\sqrt{\lambda_{k-1}/(2l_k)}$. The coefficient $1/(\sqrt{2}L_k)$ aligns with the refined AdGD method in \cite{malitsky2024adaptive}, while the additional term is our novel contribution for nonconvex settings. The second and third terms in the minimum operation in Line 8 of Algorithm \ref{adagdnc} are derived from our convergence analysis, and we note that their specific form admits flexibility, as will be elaborated in Remark \ref{rmk-general-ss}.

\section{Convergence analysis}\label{section-analysis}

In this section, we carry out the convergence analysis of Algorithm \ref{adagdnc}.
First, we present some preliminaries for this analysis, including establishing the boundedness of the step sizes $\{\lambda_k\}$, recalling useful lemmas from the literature, and introducing a Lyapunov function (a key component for the complexity analysis). 
Then, we derive the convergence results for the general nonconvex case, followed by additional results when the function $f$ is convex. For brevity, in what follows, we will not repeatedly mention the standing Assumption \ref{assumption-composite} and will take its validity for granted.
 
\subsection{Preliminaries}

\begin{proposition}
    [boundedness of $\{\lambda_k\}$]\label{prop-step-bound}
    Let $\{\lambda_k\}$ be the sequence of step sizes generated by Algorithm \ref{adagdnc}, and define $P:=\sum_{k=0}^{\infty} \rho_k$. Then, $\{\lambda_k\}$ is bounded below and above by positive constants $\lambda>0$ and $\Lambda>0$, respectively. Specifically, we have
    \begin{align}\label{step-bound}
        \lambda := \min\big(\lambda_0, 1/(2L)\big)    
        \leq \lambda_k \leq 
        \Lambda :=\lambda_0\exp(P/2), \quad\forall k\geq 0. 
    \end{align}
\end{proposition}

\begin{proof}
    We first establish the upper bound. From Algorithm \ref{adagdnc}, we have $\lambda_{i+1}/\lambda_i\leq \sqrt{1+\rho_i}$ for all $i\geq 0$, and thus $\ln \lambda_{i+1}  - \ln \lambda_i  = \ln(\lambda_{i+1}/\lambda_i)  \leq  \ln\sqrt{1+\rho_i} \leq  \rho_i/2$. Summing this inequality over $i = 0, \ldots, k$ and recalling that $P=\sum_{i=0}^{\infty} \rho_i$, we obtain $\ln\lambda_{k+1} - \ln\lambda_0 \leq \sum_{i=0}^{k}\rho_i/2 \leq P/2$. This further implies $\lambda_{k} \leq \lambda_0 \exp(P/2)$ for all $k\geq 0$. Next, we establish the lower bound by induction. The base case: $\lambda_0 \geq \lambda$ holds trivially by the definition of $\lambda$. For the inductive step, suppose that $\lambda_k \geq \lambda$ holds for some integer $k\geq 0$; we then show this inequality must also hold for $k+1$. We divide the proof into three cases:
    (i) If $\lambda_{k+1} = \sqrt{1+\rho_k}\lambda_k$, then $\lambda_{k+1} \geq \lambda_k \geq \lambda$ (the second inequality follows from the inductive hypothesis). 
    (ii) If $\lambda_{k+1} = 1/(\sqrt{2}L_{k+1})$ or $\lambda_{k+1} = 1/L_{k+1}$, then recalling that $L_k \leq L$ for all $k\geq 0$, we have $\lambda_{k+1} \geq 1/(\sqrt{2}L) \geq \lambda$. 
    (iii) If $\lambda_{k+1} = \sqrt{\lambda_k/(2l_{k+1})}$ (which possibly applies only when $l_{k+1} > 0$),  
    then by the inductive hypothesis, together with $l_{k+1} \leq  l \leq L$ and the definition of $\lambda$, we obtain $\lambda_{k+1} \geq \sqrt{ \lambda / (2L)} \geq \lambda$.  Therefore, $\lambda_k \geq \lambda$ holds for all $k\geq 0$. 
\end{proof}

An equivalent form of the proximal gradient step 
$x_{k+1} = \text{prox}_{\lambda_k h}(x_k - \lambda_k \nabla f(x_k))$ is given by
\begin{align}\label{proxequiv}
    x_{k+1} = x_k - \lambda_k(\nabla f(x_k)+h'(x_{k+1})) \text{~~for some~~}h'(x_{k+1})\in\partial h(x_{k+1}). 
\end{align}
In this paper, we employ the following gradient mapping to quantify the accuracy of an approximate solution:
\begin{align}\label{def:G}
    \CG_\eta(x): =  \left(x-\text{prox}_{\eta h}(x-\eta\nabla f(x))\right) / \eta, \text{~~for some~~}\eta>0.
\end{align}
This measure is commonly used in the context of composite optimization problems, see, e.g., \cite{nesterov2018lectures,hoaia2025composite}.
When $h \equiv 0$, we have $\CG_\eta(x) = \nabla f(x)$ for any $\eta > 0$. 
Moreover, it follows directly from \eqref{proxequiv} and the definition of $\CG_\eta(x)$ in \eqref{def:G} that $\|\CG_{\lambda_k}(x_k)\|=\|(x_{k+1}-x_{k})/\lambda_k\|=\|\nabla f(x_k)+h'(x_{k+1})\|$. In this paper, we call a point $\hat{x}\in\BR^n$ as an $\varepsilon$-stationary point if $\| \CG_\Lambda(\hat{x}) \| \leq \varepsilon$ for some $\varepsilon > 0$, where $\Lambda$ is defined in \eqref{step-bound}. 

Next, we recall two useful lemmas from the literature, both of which hold under conditions (i) and (ii) of our  Assumption \ref{assumption-composite}. 

\begin{lemma}[see {\cite[Lem. 12]{malitsky2024adaptive}}]\label{lemma-prox-ieq}
    For iterates $\{x_k\}$ of the proximal gradient method with arbitrary step sizes $\{\lambda_k\}$, it holds
    \begin{align}\label{prox-ieq}
        \|\nabla f(x_k)+h'(x_{k+1})\|\leq \|\nabla f(x_k)+h'(x_k)\|.
    \end{align}
\end{lemma}

\begin{lemma}[see {\cite[Thm. 10.9]{beck2017first}}]\label{lemma-gradient-mapping-monotone}
    For any $\eta_1 \geq \eta_2 >0$, we have $\|\CG_{\eta_1}(x)\| \leq \|\CG_{\eta_2}(x)\|$ for all $x\in\BR^n$.
\end{lemma}

For simplicity, in the rest of this section, we denote $\CG_k := \CG_{\lambda_k}(x_k)$. In adaptive algorithms, the sequence of objective function values $\{F(x_k)\}$ is not necessarily monotone. To analyze the convergence rate of AdaPGNC, we construct a Lyapunov function that incorporates both the objective function value gap and the iterate movement, defined as follows:
\begin{align}\label{lyapunov-func}
    V_k := \omega_k\Big(F(x_k) - F_* + \frac{\lambda_{k-1}}{2\lambda_k^2}\|x_{k+1}-x_k\|^2\Big) \geq 0, &\quad  \forall k\geq 0, 
\end{align}
where (for convenience, we define $\lambda_{-1}=\lambda_0$ and $\rho_{-1} = 0$)
\begin{align}\label{lyapunov-coef}
    \omega_k := 
    \frac{\lambda_k^2}{\lambda_0^2 \prod_{i=1}^{k}(1+\rho_{i-1}) \prod_{i=1}^{k}\sqrt{1+\rho_{i-2}}}, \quad \forall k\geq 1 \text{ and }\omega_0 = 1. 
\end{align}
Next, we prove some basic properties of $\{\omega_k\}$.

\begin{lemma}[properties of $\{\omega_k\}$]\label{lemma-omega-prop}
    Let $\{\omega_k\}$ be defined as in \eqref{lyapunov-coef}. Then,   we have
    \begin{align}\label{lyapunov-coef-induce}
        \omega_{k+1}   \lambda_{k}^2 (1+\rho_{k})\sqrt{1+\rho_{k-1}}= \omega_{k} \lambda_{k+1}^2,\quad \forall k\geq 0. 
    \end{align}
    Moreover, $\{\omega_k\}$ is bounded below: $\omega_k \geq \omega := 
    (\lambda^2/\lambda_0^2)\exp\left(-3P/2\right) > 0$ for all $k\geq 1$.
\end{lemma}

\begin{proof}
Fix $k\geq 0$ arbitrarily. 
Equality \eqref{lyapunov-coef-induce} follows directly from the definition of $\omega_k$ in \eqref{lyapunov-coef}. Furthermore,  by leveraging the inequality $\ln(1+x)\leq x$ (which holds for all $x\geq 0$) and the definition of $P$, we obtain
    \begin{align*}
        \prod\nolimits_{i=1}^{k}(1+\rho_{i-1}) = \exp\Big(\sum\nolimits_{i=1}^{k}\ln(1+\rho_{i-1})\Big) \leq \exp\Big(\sum\nolimits_{i=1}^{k}\rho_{i-1}\Big)\leq \exp(P).
    \end{align*}
    Similarly, we have $\prod_{i=1}^{k}\sqrt{1+\rho_{i-2}} \leq \exp(P/2)$. 
    Further, by noting that $\lambda_k\geq \lambda$ for all $k$, we deduce from \eqref{lyapunov-coef} that $\omega_k \geq (\lambda^2/\lambda_0^2)\exp\left(-3P/2\right)$.
    This completes the proof.
\end{proof}

\subsection{Convergence analysis for the nonconvex setting}\label{sc:noncv}

In this subsection, we establish convergence results for the general nonconvex setting. 
To begin with, we present a technical lemma, whose proof follows from straightforward calculations.

\begin{lemma}\label{lemma-iterate-bound}
    Let $\{x_k\}$ and $\{\lambda_k\}$ denote the sequences generated by Algorithm \ref{adagdnc}. Then, for all $k\geq 1$, it holds that
    \begin{align}\label{lyapunov-former}
        \|x_{k+1} - x_k\|^2 \leq \Big(\lambda_k^2L_k^2 + \frac{\lambda_k^2}{\lambda_{k-1}}l_k\Big)\|x_{k}-x_{k-1}\|^2 + \frac{2\lambda_k^2}{\lambda_{k-1}}\big(F(x_{k-1}) - F(x_k)\big) - \lambda_k^2 \|\CG_{k-1}\|^2. 
    \end{align}
\end{lemma}

\begin{proof}
It follows from \eqref{hestimate} and \eqref{proxequiv} that
\begin{align}\label{prox-l-smooth}
    f(x_{k-1}) & \overset{\eqref{hestimate}}= f(x_k) + \langle \nabla f(x_k),x_{k-1}-x_k\rangle - \frac{l_k}{2}\|x_k-x_{k-1}\|^2\nonumber\\
    &\overset{\eqref{proxequiv}}= f(x_k) + \lambda_{k-1} \langle \nabla f(x_k),\nabla f(x_{k-1})+h'(x_k)\rangle-\frac{l_k}{2}\|x_k-x_{k-1}\|^2.
\end{align}
Since $h'(x_k)\in\partial h(x_k)$, the subgradient inequality for the convex function $h$ implies that 
\begin{align}\label{h-convex-ieq}
    h(x_{k-1}) &\geq h(x_k) + \langle h'(x_k),x_{k-1}-x_k\rangle \overset{\eqref{proxequiv}} = h(x_k)+\lambda_{k-1}\langle h'(x_k),\nabla f(x_{k-1})+ h'(x_k)\rangle.
\end{align}
Combining \eqref{prox-l-smooth} and \eqref{h-convex-ieq}, we obtain
\begin{align}\label{prox-lk-ieq}
    F(x_{k-1}) \geq F(x_{k}) + \lambda_{k-1}\langle \nabla f(x_k) +  h'(x_k),\nabla f(x_{k-1})+ h'(x_k)\rangle-\frac{l_k}{2}\|x_k-x_{k-1}\|^2.
\end{align}
Then, we can upper bound $\|x_{k+1} - x_k\|^2$ as:
\begin{align}\label{jy-add-1}
    & \|x^{k+1}-x^{k}\|^2  
    \overset{\eqref{proxequiv}} =  \lambda_k^2\|\nabla f(x_k)+ h'(x_{k+1})\|^2 
    \overset{\eqref{prox-ieq}}\leq \lambda_k^2\|\nabla f(x_k)+ h'(x_k)\|^2\nonumber\\
    = \; &\lambda_k^2\|\nabla f(x_k)-\nabla f(x_{k-1})\|^2 - \lambda_k^2\|\CG_{k-1}\|^2 + 2\lambda_k^2\langle \nabla f(x_k)+ h'(x_k), \nabla f(x_{k-1})+ h'(x_k)\rangle \nonumber\\ 
    \overset{\eqref{hestimate}}{=\;}  & \lambda_k^2 L_k^2\|x_k-x_{k-1}\|^2- \lambda_k^2\|\CG_{k-1}\|^2 +2\lambda_k^2\langle \nabla f(x_k)+ h'(x_k), \nabla f(x_{k-1})+ h'(x_k)\rangle \nonumber\\ 
    \overset{\eqref{prox-lk-ieq}}{\leq} & \lambda_k^2 L_k^2\|x_k-x_{k-1}\|^2- \lambda_k^2\|\CG_{k-1}\|^2 + \frac{2\lambda_k^2}{\lambda_{k-1}}\Big(F(x_{k-1})-F(x_k)+\frac{l_k}{2}\|x_k-x_{k-1}\|^2\Big) \nonumber \\
    = \; & \text{right-hand side of~} \eqref{lyapunov-former}, 
\end{align}
where the second  equality follows from $\CG_{k-1} = \CG_{\lambda_{k-1}}(x_{k-1}) 
= \nabla f(x_{k-1}) + h'(x_{k})$. 
This completes the proof. 
\end{proof}

Now, we are ready to present the main convergence results of AdaPGNC in the general nonconvex setting. 
\begin{theorem}\label{thm-prox-convergence}
    Let $\{x_k\}$ and $\{\lambda_k\}$ denote the sequences generated by Algorithm \ref{adagdnc}, $V_k$ be defined as in \eqref{lyapunov-func}, and $\omega_k$ be defined as in \eqref{lyapunov-coef}. Recall also that $\lambda$ and $\Lambda$ are defined in \eqref{step-bound}. Then, for all $k\geq 1$, it holds that
    \begin{align} \label{thm-ineq:descent}
        V_k \leq V_{k-1} - \frac{\omega_k\lambda_{k-1}}{2} \|\CG_{k-1}\|^2.
    \end{align}
   This further implies
    $\lim_{k\rightarrow \infty} \|\CG_{\Lambda}(x_k)\| = 0$ and 
    $\min_{0\leq i \leq k-1}\|\CG_{\Lambda}(x_i)\|^2\leq \min_{0\leq i \leq k-1}\|\CG_{i}\|^2 \leq \frac{2 V_0}{\omega \lambda k}$. 
    Hence, for any $\varepsilon > 0$, to generate a point $x$ satisfying $\|\CG_{\Lambda}(x)\|\leq \varepsilon$, the number of iterations (or, equivalently, the number of gradient evaluations) is at most $\frac{2 V_0}{\omega \lambda \varepsilon^2}$, which is of order $\CO(1/\varepsilon^2)$.
\end{theorem}

\begin{proof}
    From the definition of $\lambda_k$ in Algorithm \ref{adagdnc}, we can deduce straightforwardly that
    \begin{align}\label{step-condition}
        \lambda_k^2 L_k^2+\frac{\lambda_k^2}{\lambda_{k-1}}l_k \leq 1. 
    \end{align}
    By using \eqref{step-condition}  to amplify the first term on the right-hand side of \eqref{lyapunov-former} and rearranging the resulting terms, we deduce
    \begin{align}\label{prox-thm-pre-pre-lyapunov}
        F(x_k) + \frac{\lambda_{k-1}}{2\lambda_k^2}\|x_{k+1}-x_k\|^2 \leq F(x_{k-1}) + \frac{\lambda_{k-1}}{2\lambda_k^2}\|x_k-x_{k-1}\|^2 - \frac{\lambda_{k-1}}{2} \|\CG_{k-1}\|^2.
    \end{align} 
    Recall that  $\lambda_k \leq \sqrt{1+\rho_{k-1}} \lambda_{k-1}$ holds for all $k\geq 1$.       Let $k \geq 1$ be arbitrarily fixed. 
    We next split the discussion into two mutually exclusive cases: (a) $\lambda_{k-1}/\lambda_k^2 \leq  \lambda_{k-2}/\lambda_{k-1}^2$, and (b) $\lambda_{k-1}/\lambda_k^2 > \lambda_{k-2}/\lambda_{k-1}^2$. 
    For case (a), we have $\lambda_{k-1}/\lambda_k^2 \leq  \lambda_{k-2}/\lambda_{k-1}^2$, which together with \eqref{prox-thm-pre-pre-lyapunov}, implies that
    \begin{align}
        &F(x_k) - F_* + \frac{\lambda_{k-1}}{2\lambda_k^2}\|x_{k+1}-x_k\|^2 \nonumber\\
        \overset{\eqref{prox-thm-pre-pre-lyapunov}, (a)}\leq \, &  F(x_{k-1}) - F_* + \frac{\lambda_{k-2}}{2\lambda_{k-1}^2}\|x_k-x_{k-1}\|^2 - \frac{\lambda_{k-1}}{2} \|\CG_{k-1}\|^2\nonumber\\
        \leq \;\; \;\;
        &\Big(F(x_{k-1}) - F_* + \frac{\lambda_{k-2}}{2\lambda_{k-1}^2}\|x_k-x_{k-1}\|^2\Big) \frac{\lambda_{k-1}^2}{\lambda_{k}^2}(1+\rho_{k-1})  \sqrt{1+\rho_{k-2}} - \frac{\lambda_{k-1}}{2} \|\CG_{k-1}\|^2. \nonumber
    \end{align}  
    For case (b), we have $\lambda_{k-1}^3/(\lambda_k^2 \lambda_{k-2}) > 1$, which together with \eqref{prox-thm-pre-pre-lyapunov}, implies that
    \begin{align}
        &F(x_k) - F_* + \frac{\lambda_{k-1}}{2\lambda_k^2}\|x_{k+1}-x_k\|^2 \nonumber\\
        \overset{\eqref{prox-thm-pre-pre-lyapunov}, (b)}\leq \, &\Big(F(x_{k-1}) - F_* + \frac{\lambda_{k-2}}{2\lambda_{k-1}^2}\|x_k-x_{k-1}\|^2 \Big)\frac{\lambda_{k-1}^3}{\lambda_k^2 \lambda_{k-2}} - \frac{\lambda_{k-1}}{2} \|\CG_{k-1}\|^2\nonumber\\
        \leq \;\;\;\;  
        &\Big(F(x_{k-1}) - F_* + \frac{\lambda_{k-2}}{2\lambda_{k-1}^2}\|x_k-x_{k-1}\|^2\Big) \frac{\lambda_{k-1}^2}{\lambda_k^2}(1+\rho_{k-1})  \sqrt{1+\rho_{k-2}} - \frac{\lambda_{k-1}}{2} \|\CG_{k-1}\|^2.\nonumber
    \end{align}
    In both cases, we arrived at the same inequality
    \begin{align}\label{prox-thm-pre-lyapunov}
        &F(x_k) - F_* + \frac{\lambda_{k-1}}{2\lambda_k^2}\|x_{k+1}-x_k\|^2 \nonumber\\
        \leq&\Big(F(x_{k-1}) - F_* + \frac{\lambda_{k-2}}{2\lambda_{k-1}^2}\|x_k-x_{k-1}\|^2\Big) \frac{\lambda_{k-1}^2}{\lambda_k^2}(1+\rho_{k-1})  \sqrt{1+\rho_{k-2}} - \frac{\lambda_{k-1}}{2} \|\CG_{k-1}\|^2. 
    \end{align}    
    Multiplying both sides of \eqref{prox-thm-pre-lyapunov} by $\omega_k$ (as defined in \eqref{lyapunov-coef}), and utilizing the definition of $V_k$ in \eqref{lyapunov-func} together with the equality in \eqref{lyapunov-coef-induce}, we derive \eqref{thm-ineq:descent} immediately as follows:
    \begin{align*}
        V_k \overset{\eqref{lyapunov-func}} = \;\;\;\; &\omega_k\Big(F(x_k) - F_* + \frac{\lambda_{k-1}}{2\lambda_k^2}\|x_{k+1}-x_k\|^2\Big)\nonumber\\
        \overset{\eqref{prox-thm-pre-lyapunov}, \eqref{lyapunov-coef-induce}}{\leq} \; &\omega_{k-1}\Big(F(x_{k-1}) - F_* + \frac{\lambda_{k-2}}{2\lambda_{k-1}^2}\|x_k-x_{k-1}\|^2\Big) - \frac{\omega_k\lambda_{k-1}}{2} \|\CG_{k-1}\|^2\nonumber\\
        \overset{\eqref{lyapunov-func}} = \;\;\;\; &V_{k-1} - \frac{\omega_k\lambda_{k-1}}{2} \|\CG_{k-1}\|^2.
    \end{align*}
    By replacing the index $k$ with $i$ in the above inequality and then summing the inequality over $i=1,\ldots,k$, we obtain
    \begin{align}\label{prox-thm-lypunov-sum}
        \sum_{i=1}^{k}\frac{\omega_{i}\lambda_{i-1}}{2} \|\CG_{i-1}\|^2\leq V_0 - V_{k} \leq V_0.
    \end{align}
    Since $\omega_{i} \geq \omega > 0$ (see Lemma \ref{lemma-omega-prop})  and $\lambda_i \geq \lambda > 0$ (see Eq. \eqref{step-bound}), it follows easily from \eqref{prox-thm-lypunov-sum} that $\lim_{k\rightarrow \infty}\|\CG_k\| = 0$ and $\min_{0\leq i \leq k-1} \|\CG_{i}\|^2 \leq \frac{2 V_0}{\omega \lambda  k}$.
    Furthermore, since $\|\CG_\Lambda(x_i)\| \leq \|\CG_{\lambda_i}(x_i)\| = \|\CG_i\|$ for all $i$, which is implied by Lemma \ref{lemma-gradient-mapping-monotone} together with the fact that $\lambda_i\leq \Lambda$, we immediately obtain $\lim_{k \rightarrow \infty} \|\CG_\Lambda(x_k)\| = 0$. The remaining claim of the theorem follows straightforwardly. 
\end{proof}

\begin{remark}
    In Algorithm \ref{adagdnc} (Lines 6 and 8), we adopt $\sqrt{1+\rho_k}$ to regulate the growth rate of the step size $\lambda_k$, but the square root form here can adopt some other choices. In fact, $\sqrt{1+\rho_k}$ can be replaced by $(1+\rho_k)^q$ for any constant $q>0$, and this modification does not undermine the algorithm's theoretical guarantees. In fact, 
    similar convergence results can still be established with only minor adjustments to the proofs. 
     Specifically, the key adjustments are twofold: first, redefine the coefficient $\omega_k$ (denoted as $\omega_{k,q}$ here to distinguish it from the original $\omega_k$) as follows:
    \begin{align*}
        \omega_{k,q} 
        = \frac{\lambda_k^2}{\lambda_0^2\prod_{i=1}^{k}(1+\rho_{i-1})^{2q} \prod_{i=1}^{k}(1+\rho_{i-2})^{q}}, \quad \forall k\geq 1. 
    \end{align*}
    Second,  update the lower bound $\lambda$ and upper bound $\Lambda$ of the step size $\lambda_k$; additionally, update the definition of $\omega$ in Lemma \ref{lemma-iterate-bound} to ensure it serves as a valid positive lower bound for $\omega_{k,q}$. 
    With these adjustments in place, the remainder of the convergence proof remains entirely unchanged. 
\end{remark}

\begin{remark}\label{rmk-general-ss}
From the above analysis, the key factors ensuring the convergence results in Theorem \ref{thm-prox-convergence} are twofold: first, the boundedness of $\{\lambda_k\}$ established in Proposition \ref{prop-step-bound}; second, the satisfaction of condition \eqref{step-condition}. Guided by this insight, we can replace the step size rule in Algorithm \ref{adagdnc} (Lines 5-9) with the following more relaxed one: $\lambda_k = \min\left\{\sqrt{1+\rho_k}\lambda_{k-1}, \big([L_k^2 + l_k/\lambda_{k-1}]_+\big)^{-\frac{1}{2}}\right\}$. 
    In particular, this relaxed rule retains two critical theoretical properties: it still guarantees the satisfaction of \eqref{step-condition}, and it can potentially yield larger step sizes than the original rule in Algorithm \ref{adagdnc}. Moreover, the boundedness of $\{\lambda_k\}$ can be established following a similar line of reasoning as in Proposition \ref{prop-step-bound}, and all subsequent theoretical results can be derived analogously under this relaxed step size rule. We did not adopt it in the main algorithm, however, because our current step size choice demonstrated better performance in our experiments.  
\end{remark}

\subsection{Convergence analysis for the convex setting}\label{sc:cv}

In this subsection, we establish additional convergence results, measured by the function value gap, for Algorithm \ref{adagdnc} under the convexity assumption of $f$. First, we derive an upper bound for $\sum_{i=1}^{k}\lambda_i^2\|\nabla f(x_i) +  h'(x_i)\|^2$; notably, this bound does not rely on the convexity of $f$. Recall that $\lambda\leq \lambda_k\leq \Lambda$ for all $k$, with $\lambda, \Lambda > 0$ as defined in \eqref{step-bound}, and $\omega_k\geq \omega>0$ for all $k$, with $\omega$ as defined in Lemma \ref{lemma-omega-prop}. 

\begin{lemma}[an upper bound for $\sum_{i=1}^{k}\lambda_i^2\|\nabla f(x_i) +  h'(x_i)\|^2$]
Let $\{x_k\}$ and $\{\lambda_k\}$ denote the sequences generated by Algorithm \ref{adagdnc}. Then, it holds that
\begin{align}    \label{proxiterate-sum-bound}
    \sum_{i=1}^{k}\lambda_i^2\|\nabla f(x_i) +  h'(x_i)\|^2
    \leq \CS := \frac{\Lambda^2}{\lambda}\Big(\frac{2\Lambda^2}{\omega\lambda^2} V_0 +  2\big(F(x_0) - F_*\big)\Big) > 0.
\end{align}
\end{lemma}

\begin{proof}
It follows from \eqref{jy-add-1} and \eqref{step-condition} that 
\begin{align}\label{jy-add-2}
    \lambda_k^2 \|\nabla f(x_k) +h'(x_k)\|^2 
    & \overset{\eqref{jy-add-1}}{\leq} \Big(\lambda_k^2 L_k^2+\frac{\lambda_k^2}{\lambda_{k-1}}l_k\Big) \|x_k-x_{k-1}\|^2 + \frac{2\lambda_k^2}{\lambda_{k-1}}\big(F(x_{k-1}) - F(x_k)\big) \nonumber \\
    & \overset{\eqref{step-condition}}{\leq} \|x_k-x_{k-1}\|^2 + \frac{2\lambda_k^2}{\lambda_{k-1}}\big(F(x_{k-1}) - F(x_k)\big). 
\end{align} 
Multiplying both sides of \eqref{jy-add-2} by $\lambda_{k-1}/\lambda_{k}^2$, and utilizing \eqref{proxequiv} along with the relation $\CG_{k-1} = \nabla f(x_{k-1}) + h'(x_{k})$, we can reformulate \eqref{jy-add-2} equivalently  as
\begin{align*}
    \lambda_{k-1}\|\nabla f(x_k) + h'(x_k)\|^2 \leq \frac{\lambda_{k-1}^3}{\lambda_k^2}\|\CG_{k-1}\|^2+2\big(F(x_{k-1}) -F(x_k)\big).
\end{align*}
By replacing the index $k$ with $i$ in the above inequality, then telescoping the resulting inequality over $i = 1, \ldots, k$, and further utilizing $F(x_k) \geq F_*$, we derive
\begin{align}\label{jy-add-3}
    \sum_{i=1}^k\lambda_{i-1}\|\nabla f(x_{i})+ h'(x_{i})\|^2
    \leq \; & \sum_{i=1}^k \frac{\lambda_{i-1}^3}{\lambda_{i}^2}\|\CG_{i-1}\|^2 + 2 \big(F(x_0) - F_*\big)\nonumber \\
    \overset{\eqref{prox-thm-lypunov-sum}}{\leq} & 
    C := 
    \frac{2\Lambda^2}{\omega\lambda^2} V_0 +  2\big(F(x_0) - F_*\big), 
\end{align}
where the second inequality also leverages the bounds $\lambda\leq \lambda_i\leq \Lambda$ and $\omega_i\geq \omega$ (holding for all $i$). Reapplying the bounds $\lambda\leq \lambda_i\leq \Lambda$ (valid for all $i$), we can further derive
\begin{align*}
    \sum_{i=1}^{k}\lambda_i^2\|\nabla f(x_i) +  h'(x_i)\|^2
    = & \sum_{i=1}^k \frac{\lambda_i^2}{\lambda_{i-1}}\cdot\lambda_{i-1}\|\nabla f(x_i)+ h'(x_i)\|^2 
    \leq  \frac{\Lambda^2 C}{\lambda} = \CS, 
\end{align*}
where $C>0$ is defined as in \eqref{jy-add-3} and $\CS$ is defined in \eqref{proxiterate-sum-bound}. This completes the proof. 
\end{proof}

Next, we establish additional convergence rate results by leveraging the bound given in  \eqref{proxiterate-sum-bound}, under the convexity assumption on $f$.

\begin{theorem}\label{prox-thm-convex-convergence}
    Assume that $f$ is convex and that the set of minimizers of $F$ is nonempty.  Let $\{x_k\}$ denote the sequence generated by Algorithm \ref{adagdnc}. Then, we have
    \begin{align*}
        \min_{1\leq i\leq k} F(x_i) - F_* \leq \frac{\|x_1-x_*\|^2 + \CS}{2\lambda k } = \CO\left(\frac{1}{k}\right),
    \end{align*}
    where  $x_*\in\BR^n$ is any minimizer of $F$, i.e., $F(x_*) = F_*$, and 
    $\mathcal{S} > 0$ is as defined in \eqref{proxiterate-sum-bound}. Moreover, we have the following ergodic convergence rate result:
    \begin{align}\label{jy-add-7}
        F(\bar{x}_k) - F_* \leq \frac{\|x_1-x_*\|^2 + \CS}{2\lambda k } = \CO\left(\frac{1}{k}\right), \text{ where }\bar{x}_k = \frac{\sum_{i=1}^{k}\lambda_i x_i}{\sum_{i=1}^{k}\lambda_i}.
    \end{align}
    This implies that to guarantee $F({\bar x}_k) - F_* \leq \varepsilon$ for any given $\varepsilon >0$, no more than $\CO(1/\varepsilon)$ iterations are required. 
\end{theorem}

\begin{proof}
    By utilizing the convexity of $h$, we have
    \begin{align}\label{jy-add-4}
        \lambda_k (h(x_{k+1})- h(x_*))\leq \lambda_k \langle  h'(x_{k+1}), x_{k+1} - x_*\rangle 
        \overset{\eqref{proxequiv}}{=} \langle x_k - x_{k+1} - \lambda_k \nabla f(x_k), x_{k+1} - x_*\rangle.
    \end{align}    
    Multiply both sides of \eqref{jy-add-4} by $2$ and add it to the following identity
    \begin{align*}
        \|x_{k+1}-x_*\|^2 
        = \|x_k-x_*\|^2 - \|x_{k+1}-x_k\|^2 + 2\langle x_{k+1} - x_*, x_{k+1}-x_k\rangle,
    \end{align*}
    we deduce
    \begin{align}\label{prox-thm-convex-ieq-one}
        \|x_{k+1} - x_*\|^2 + 2\lambda_k (h(x_{k+1}) - h(x_*)) \leq 
        \|x_k-x_*\|^2 - \|x_{k+1}-x_k\|^2  + 2\lambda_k \langle \nabla f(x_k), x_*-x_{k+1}\rangle.
    \end{align}
    Next, we estimate the term $\langle \nabla f(x_k), x_*-x_{k+1}\rangle$ in \eqref{prox-thm-convex-ieq-one} by
    \begin{align}\label{jy-add-5}
        \langle \nabla f(x_k), x_*-x_{k+1}\rangle 
        &= \langle \nabla f(x_k), x_* - x_k\rangle + \langle \nabla f(x_k) +  h'(x_{k}), x_k - x_{k+1}\rangle - \langle  h'(x_{k}),  x_k - x_{k+1}\rangle \nonumber \\
        &\leq f(x_*) - f(x_k) + \langle \nabla f(x_k) +  h'(x_{k}), x_k - x_{k+1}\rangle + h(x_{k+1}) - h(x_{k}), 
    \end{align}
    where the inequality follows from the convexity of $f$ and $h$. 
    Plugging \eqref{jy-add-5} into \eqref{prox-thm-convex-ieq-one}, rearranging the terms, and using $F = f + h$, we obtain
    \begin{align*}
        \|x_{k+1} - x_*\|^2 + 2\lambda_k(F(x_{k}) - F_*) 
        \leq \, & \|x_k - x_*\|^2 - \|x_{k+1} - x_k\|^2 + 2\lambda_k\langle \nabla f(x_k) +  h'(x_{k}), x_k - x_{k+1}\rangle \nonumber\\
        \leq \, & \|x_k - x_*\|^2 + \lambda_k^2 \|\nabla f(x_k) +  h'(x_k)\|^2,
    \end{align*}
    where the second inequality here follows from $2\langle u, v\rangle \leq \|u\|^2 + \|v\|^2$. 
    By replacing the index $k$ with $i$ in the above inequality, then telescoping the resulting inequality over $i = 1, \ldots, k$, we deduce
    \begin{align}\label{thm-prox-convex-sum}
        2\sum_{i=1}^k \lambda_i (F(x_i)-F_*) \leq  \|x_1-x_*\|^2 + \sum_{i=1}^{k}\lambda_i^2 \|\nabla f(x_i)+ h'(x_i)\|^2 \overset{\eqref{proxiterate-sum-bound}}{\leq} & \|x_1-x_*\|^2 + \CS.
    \end{align}
    It then follows straightforwardly from \eqref{thm-prox-convex-sum} that
    \begin{align}\label{jy-add-6}
        \min_{1\leq i\leq k} F(x_i) - F_* \leq \frac{\|x_1-x_*\|^2 + \CS}{2\sum_{i=1}^{k} \lambda_i} \leq \frac{\|x_1-x_*\|^2 + \CS}{2\lambda k } = \CO\left(\frac{1}{k}\right).
    \end{align}
    Furthermore, following a similar derivation to that of \eqref{jy-add-6}, but instead leveraging the convexity of $F$ and Jensen's inequality, we can derive 
    the ergodic convergence rate result given in \eqref{jy-add-7} from \eqref{thm-prox-convex-sum}. We omit the details for brevity, and this concludes the proof. 
\end{proof}

\begin{remark}
Let $\varepsilon>0$. 
From \eqref{jy-add-6}, it can be inferred that choosing $k= \lceil
(\|x_1-x_*\|^2 + \CS)/(2\lambda\varepsilon) \rceil$ is sufficient to guarantee  $\min_{1\leq i \leq k}F(x_i) - F_* \leq \varepsilon$. In other words, to ensure the existence of some index $i_0 \in \{1,\ldots, k\}$ satisfying $F(x_{i_0}) - F_*\leq \varepsilon$, no more than  $\CO(1/\varepsilon)$ iterations are needed, a complexity bound directly implied by the expression for $k$. 
\end{remark}

\begin{remark}
    When the objective function $f$ is convex, we always have $l_k\leq 0$. Leveraging this property, we can incorporate the short Barzilai-Borwein (BB) step size (cf.  \cite{barzilai1988two,zhou2024adabb}) into the design of a novel step size formula, given by
    \begin{align}\label{bb-step}
    \lambda_k = \min\left\{\sqrt{1+\rho_k} \lambda_{k-1}, \lambda_k^{BB}\right\} 
    \text{~~with~~} \lambda_k^{BB} := \frac{\langle \nabla f(x_k) - \nabla f(x_{k-1}), x_k - x_{k-1}\rangle}{\|\nabla f(x_k) - \nabla f(x_{k-1})\|^2}.
    \end{align}
    The theoretical validity of this step size choice is justified by the following two key observations: 
    (i) Owing to the Lipschitz continuity of $\nabla f$ and the convexity of $f$,  $\lambda_k^{BB}$ is inherently bounded below by $1/L$ 
    (see, e.g., \cite[Thm. 2.1.5]{nesterov2018lectures}). 
    Then, by incorporating this observation into the proof of Proposition \ref{prop-step-bound}, we can show that $\min\{\lambda_0, 1/L\} \leq \lambda_k \leq \Lambda$. 
    (ii) By the Cauchy-Schwarz inequality, $\lambda_k$ always satisfies
    $\lambda_k \leq \lambda_k^{BB} \leq \|x_k - x_{k-1}\|/\|\nabla f(x_k) - \nabla f(x_{k-1})\| = 1/L_k$. Given that $l_k\leq 0$, the step size $\lambda_k$ automatically satisfies the general step size condition specified in \eqref{step-condition}. 
    This extension underscores the flexibility of condition \eqref{step-condition} with respect to step size selection. 
\end{remark}

\section{Numerical experiments}\label{sectionnumericalexperiments}
In this section, we evaluate the performance of our proposed AdaPGNC algorithm and variants on five benchmark optimization problems, comparing them with several state-of-the-art parameter-free methods. Of the five test problems, four are nonconvex, and one is the convex logistic regression problem. All experiments were implemented in Python on a workstation equipped with an AMD Ryzen 9 5900HX CPU and 16GB RAM. 

In Section \ref{section-introduction}, we reviewed various parameter-free optimization algorithms for nonconvex settings. Among these, NASCAR \cite{lan2024optimalparameterfreegradientminimization} and PF.APD \cite{kong2024complexityoptimalparameterfreefirstordermethods} feature intricate multi-loop structures, rendering them unsuitable for general applications. On the other hand, AC-PG \cite{lan2024projected} requires the problem being solved to have a bounded domain---this is inconsistent with our study and experimental focus. Thus, we do not compare our algorithms with NASCAR, PF.APD and AC-PG.  Our experiments proceed as follows.
\begin{itemize}
\item First, we compare AdaPGNC with AGDR  in \cite[Alg. 4.1]{marumo2024parameterfree} and GD-LS (gradient descent with line-search)  on three nonconvex problems satisfying Assumption \ref{assumption-composite}. 

\item Second, we test a nonnegative matrix factorization (NMF) problem that does not satisfy Assumption \ref{assumption-composite}. For this NMF problem, we compare AdaPGNC only with NPG1 and NPG2:  \cite{hoaia2025composite} reports that NPG1 and NPG2 outperform line-search techniques, and AGDR is unsuitable for solving this problem as it requires estimating the Lipschitz constant of $\nabla^2 f$ via $M_k$, which, by its definition in \cite{marumo2024parameterfree}, can diverge to infinity. 

\item Finally, we evaluate the BB-type variant of AdaPGNC given in \eqref{bb-step}  against   AdaBB and AdaBB3 proposed in \cite{zhou2024adabb} on the convex logistic regression problem. 
\end{itemize}

For AdaPGNC, we employ two distinct sequences ${\rho_k}$, inspired by 
\cite{malitsky2019adaptive,hoai2024novel,hoaia2025composite}:
\begin{align*}
\rho_k^{(1)} = \min\left\{\frac{\lambda_k}{\lambda_{k-1}}, \frac{100 (\ln (k+1))^4}{(k+1)^{1.1}}\right\} \quad \text{and} \quad \rho_k^{(2)} = \frac{100 (\ln (k+1))^4}{(k+1)^{1.1}}, \quad \forall k \geq 1,
\end{align*}
with $\rho_0 = 10^{10}$ for both cases. A large $\rho_0$ permits virtually unrestricted initial growth of $\lambda_1$, enabling rapid adjustment to local curvature. The initial step size $\lambda_0$ is problem-dependent and will be specified later. For direct comparability, our algorithm was integrated into established experimental frameworks: the codebase from \cite{marumo2024parameterfree, hoaia2025composite} for nonconvex test problems and from \cite{malitsky2019adaptive} for the convex logistic regression problem. For unbiased comparison, all algorithmic parameters under consideration were adjusted to the most stable and efficient values reported in their original publications \cite{marumo2024parameterfree, zhou2024adabb, hoaia2025composite}. The details of all the algorithms compared are summarized in Table \ref{tab:algorithm_config}.

\begin{table}[h!]
\centering
\caption{Summary of compared algorithms and their parameter configurations.}
\begin{tabular}{m{3.5cm} m{10cm}} 
\toprule
\textbf{Algorithm} & \textbf{Description and Parameters} \\
\midrule
\textbf{AdaPGNC-1} & Algorithm~\ref{adagdnc} with $\rho_k = \rho_k^{(1)}$ for $k \geq 1$ and  $\rho_0 = 10^{10}$. \\
\textbf{AdaPGNC-2} & Algorithm~\ref{adagdnc} with $\rho_k = \rho_k^{(2)}$ for $k \geq 1$ and  $\rho_0 = 10^{10}$. \\
\midrule
\textbf{GD-LS} & Gradient descent with Armijo line-search. Backtracking finds the smallest nonnegative integer $m_k \geq 0$ such that: \newline $f(x_k - 10^{-3}\cdot 2^{-m_k} \nabla f(x_k)) \leq f(x_k) - 10^{-3} \cdot 2^{-(m_k+1)}\|\nabla f(x_k)\|^2$. \\
\midrule
\textbf{AGDR} &  \cite[Alg. 4.1]{marumo2024parameterfree} with  $(L_{\text{init}}, M_0, \alpha, \beta)=(10^{-3}, 10^{-16}, 2, 0.9)$. \\
\midrule
\textbf{NPG1} & \cite[Alg. 3.1]{hoaia2025composite} with $\gamma_{k-1} = 0.1(\ln k)^{5.7} / k^{1.1}$ for $k\geq 1$, $c_0 = 0.7$, and $c_1 = 0.69$. \\
\textbf{NPG2} & \cite[Alg. 4.1]{hoaia2025composite} with $\gamma_{k-1} = 0.1(\ln k)^{5.7} / k^{1.1}$ for $k\geq 1$, $c_0 = 0.99$, and $c_1 = 0.98$. \\
\midrule
\textbf{AdaPGNC-BB-1} & Algorithm \ref{adagdnc} with step size \eqref{bb-step}, $\rho_k = \rho_k^{(1)}$ for $k\geq 1$,  and $\rho_0 = 10^{10}$. \\
\textbf{AdaPGNC-BB-2} & Algorithm \ref{adagdnc} with step size \eqref{bb-step}, $\rho_k = \rho_k^{(2)}$ for $k\geq 1$, and $\rho_0 = 10^{10}$.\\
\midrule
\textbf{AdaBB} \newline \textbf{AdaBB3} & The two top-performing variants of AdaBB in \cite[Table 2]{zhou2024adabb}, implemented with their original parameters.\\
\bottomrule
\end{tabular}

\label{tab:algorithm_config}
\end{table}

\subsection{Experiments on four nonconvex problems}\label{sec:experiment-nonconvex}

This subsection evaluates the performance of the algorithms on four nonconvex problems: a classification problem, an autoencoder training problem, a low-rank matrix completion problem, and the NMF problem. For the first three problems, we set the initial step size $\lambda_0 = 1$ for AdaPGNC-1 and AdaPGNC-2. We report the minimal function values achieved within time $T$, defined as 
\begin{align*}
    f_T =& \text{ the minimal function value } f(x_k) \text{ until time } T,\\
    g_T =& \text{ the minimal gradient norm } \|\nabla f(x_k)\| \text{ until time } T,
\end{align*} 
and plot them against runtime. For the NMF problem, we set a uniform initial step size of $\lambda_0 = 0.001$ for AdaPGNC-1, AdaPGNC-2, and NPG. Details of each problem are elaborated below.

\medskip 

\noindent\textbf{Classification problem.} We consider the nonconvex classification problem:
\begin{align}\label{classification}
\min_{\omega\in\mathbb{R}^n} f(\omega) = \frac{1}{N}\sum_{i=1}^N\ell_{CE}(y_i,\phi(z_i;\omega)),
\end{align}
which minimizes the cross-entropy loss over a neural network. The dataset comprises $N = 10,\!000$ samples randomly selected from the MNIST dataset, where inputs $\{z_i\}_{i=1}^N \subset \mathbb{R}^{784}$ are images and outputs $\{y_i\}_{i=1}^N \subset \{0,1\}^{10}$ are one-hot encoded labels. The network model $\phi(\cdot;\omega): \mathbb{R}^{784} \rightarrow \mathbb{R}^{10}$ is constructed as a three-layer fully-connected neural network, incorporating bias and utilizing logistic sigmoid activation, where $\sigma(z) = 1/(1 + \exp(-z))$. The layer dimensions are 784, 32, 16, and 10, resulting in $n = 25,\!818$ trainable parameters $\omega \in \mathbb{R}^n$. The cross-entropy loss, denoted as $\ell_{CE}$, for a label $y_i = (y_{i,1},\dots,y_{i,10})^\top$ and a prediction $p_i = \phi(z_i;\omega) = (p_{i,1},\dots,p_{i,10})^\top$ is defined as 
$\ell_{CE}(y_i,p_i) = -\sum\nolimits_{k=1}^{10} y_{i,k} \log p_{i,k}$.
The default \texttt{flax.linen.Module.init} method is used for parameter initialization. 
We terminate the compared algorithms either after 500 seconds have elapsed or when the gradient norm is less than $10^{-10}$.  Figure~\ref{fig:classification-result} shows the comparison results for minimal function values, minimal gradient norms, and step sizes.

\begin{figure}[htb!]
    \centering
    \includegraphics[width=0.3\textwidth]{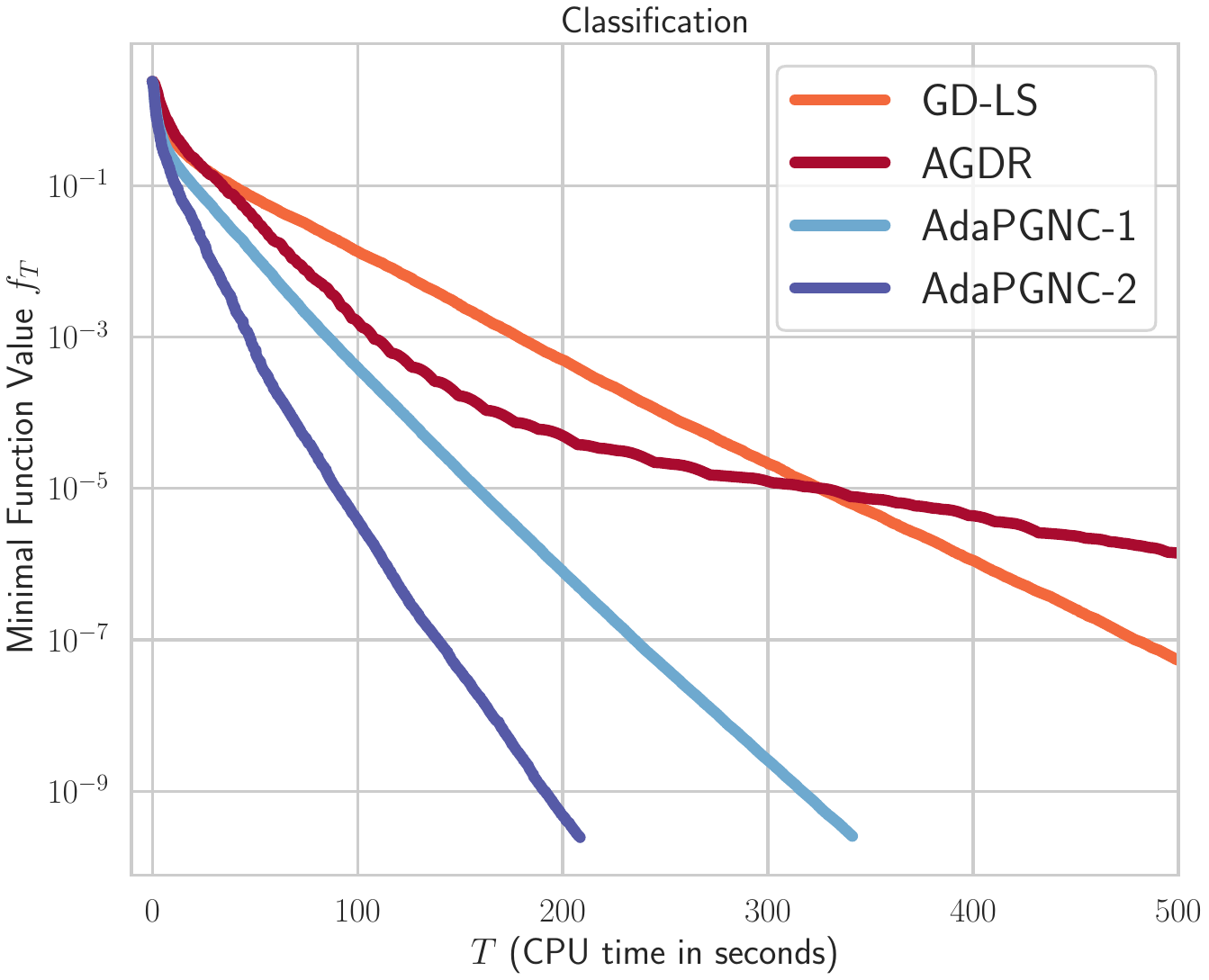}
    \includegraphics[width=0.3\textwidth]{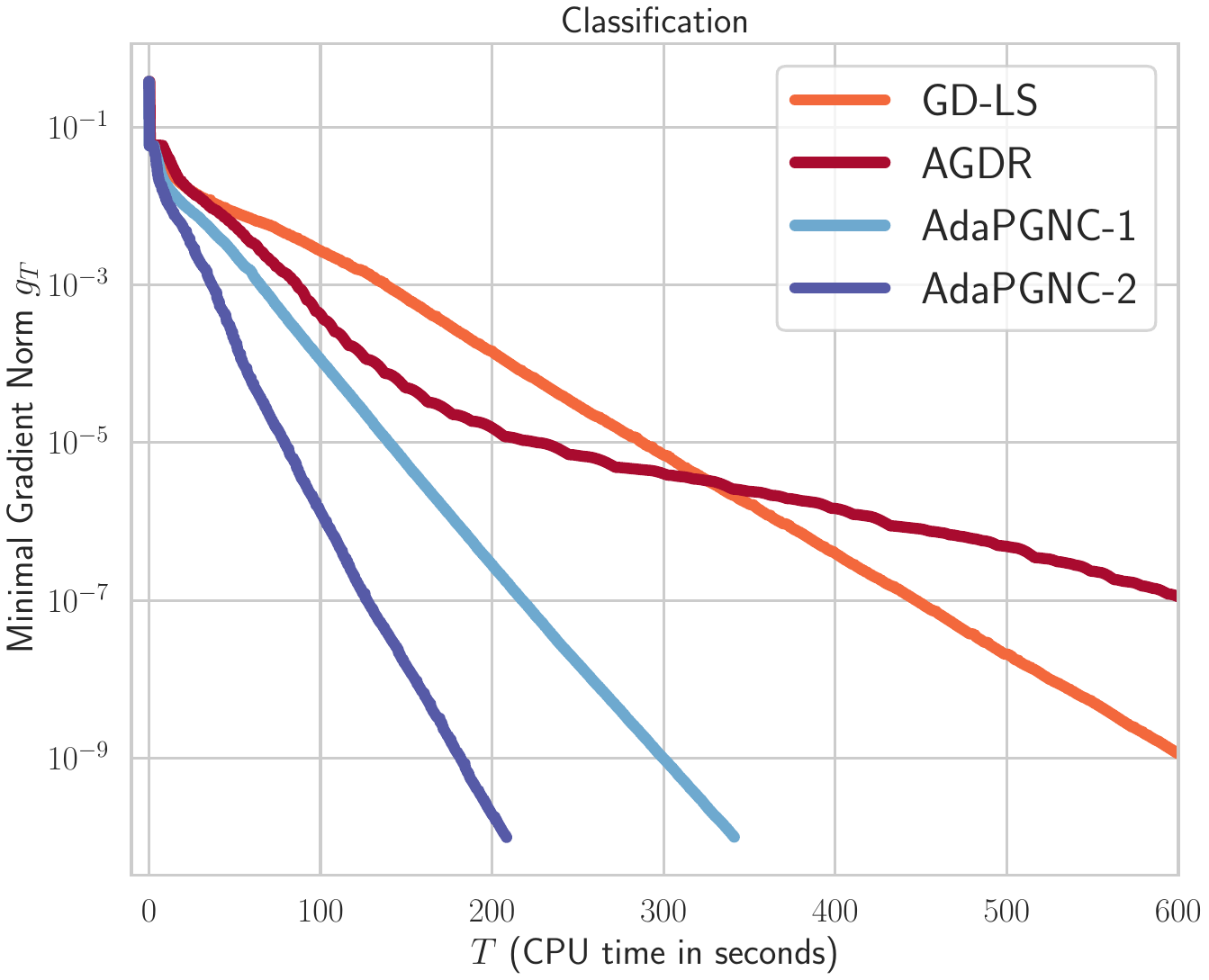}
    \includegraphics[width=0.3\textwidth]{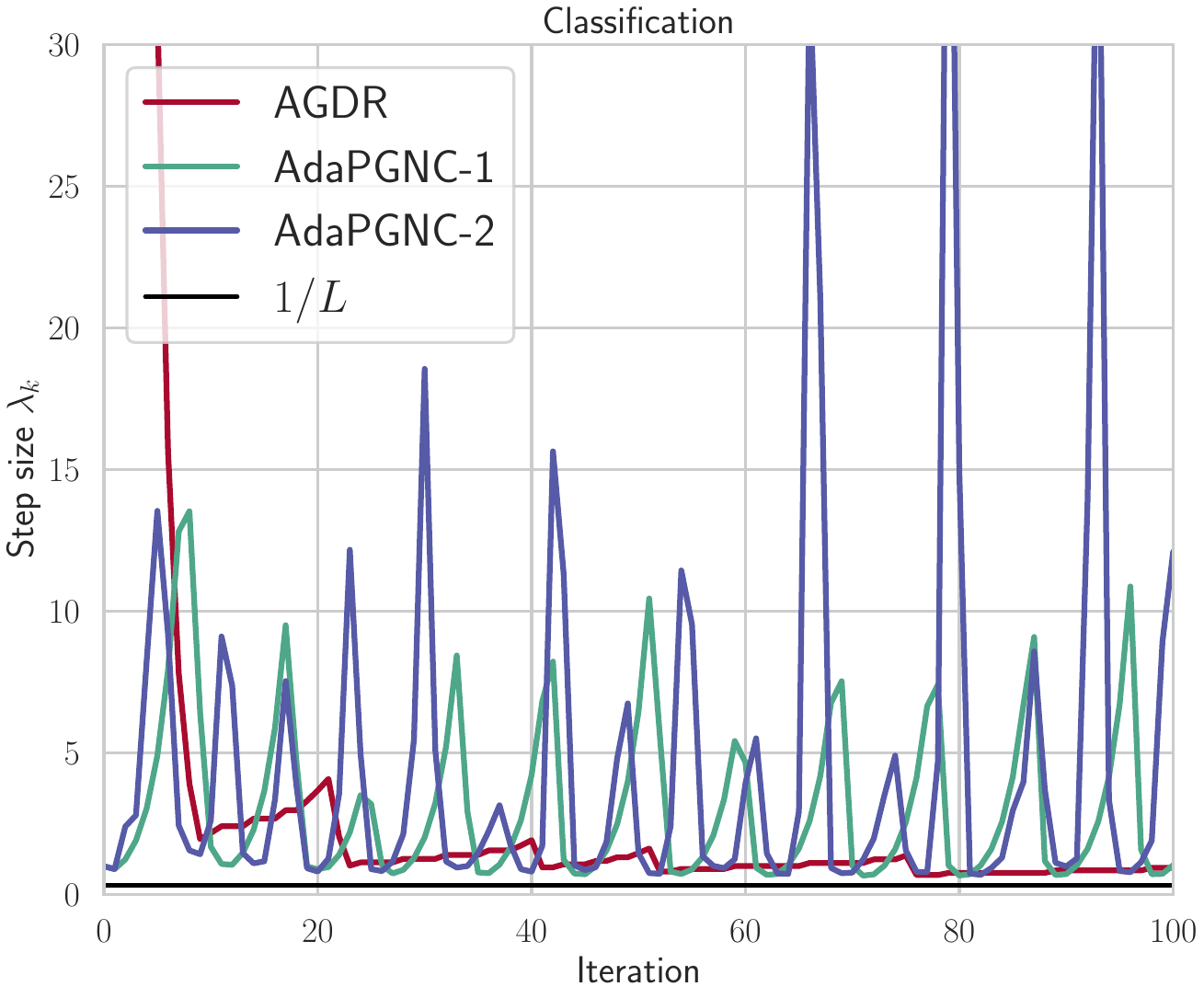}
    \caption{Comparison results for the classification problem.}
    \label{fig:classification-result}
\end{figure}

Figure \ref{fig:classification-result} shows that AdaPGNC-2 converges to the lowest objective values and gradient norms, followed by AdaPGNC-1; both outperform AGDR and GD-LS. The superior performance of AdaPGNC-2 stems from its step size behavior: as shown in the third plot, AdaPGNC-2 achieves larger step sizes than both AdaPGNC-1 and AGDR. This indicates that AdaPGNC-2's adaptive mechanism more effectively leverages favorable local curvature, enabling larger steps that contribute to its accelerated convergence.

\medskip 
\noindent\textbf{Autoencoder training problem.} We next consider the nonconvex autoencoder problem, formulated as follows:
\begin{align}\label{autoencoder}
\min_{\vartheta\in\mathbb{R}^n}
f(\vartheta) = \frac{1}{2MN}\sum_{i=1}^{N}\|u_i-\varphi(u_i;\vartheta)\|^2,
\end{align}
whose purpose is to learn a compressed data representation by minimizing the reconstruction error, measured by the squared $L^2$-norm. The model $\varphi(\cdot;\vartheta): \mathbb{R}^{784} \rightarrow \mathbb{R}^{784}$ is a four-layer fully-connected neural network with bias parameters and sigmoid activation. The layer dimensions are 784, 32, 16, 32, and 784, resulting in $n = 52,\!064$ trainable parameters $\vartheta \in \mathbb{R}^n$. We use the same MNIST dataset ($\{u_i\}_{i=1}^N \subset \mathbb{R}^{784}$ with $M=784$) as in the classification problem \eqref{classification}. The default \texttt{flax.linen.Module.init} method is used for parameter initialization. The same as before, we terminate the compared algorithms either after 500 seconds have elapsed or when the gradient norm is found to be less than $10^{-10}$.  Figure~\ref{fig:autoencoder-result} shows the detailed comparison results.

\begin{figure}[ht!]
    \centering
    \includegraphics[width=0.3\textwidth]{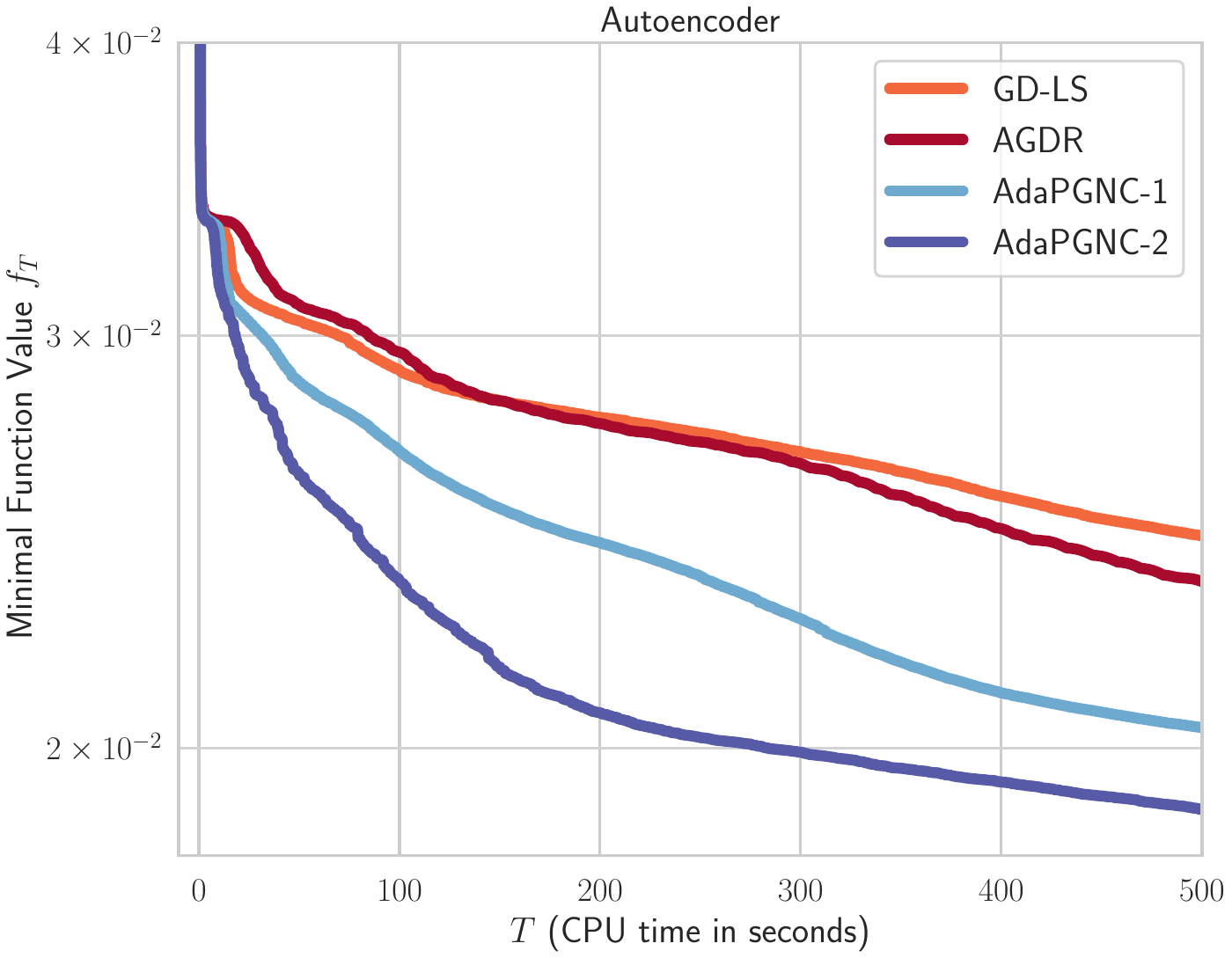}
    \includegraphics[width=0.3\textwidth]{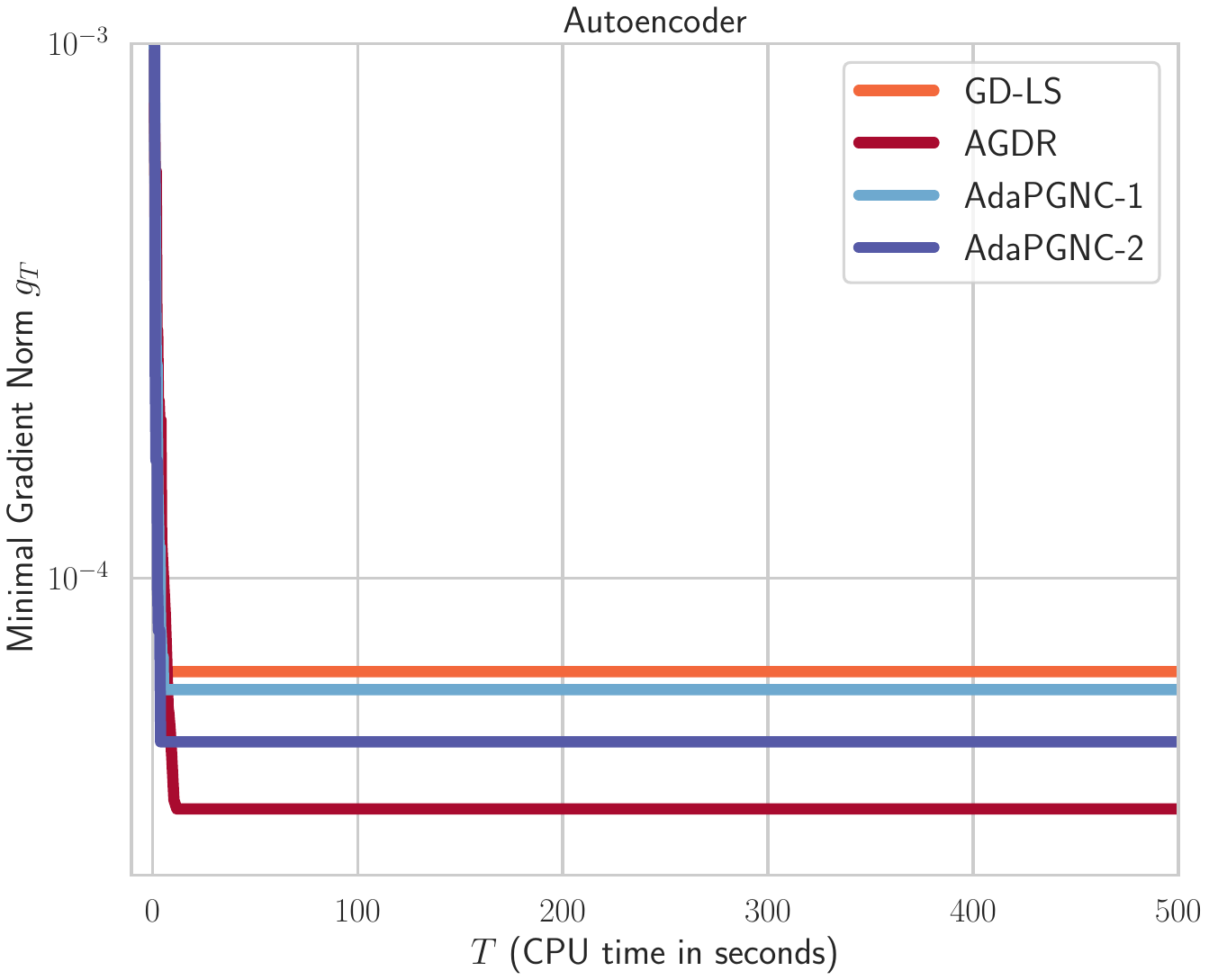}
    \includegraphics[width=0.3\textwidth]{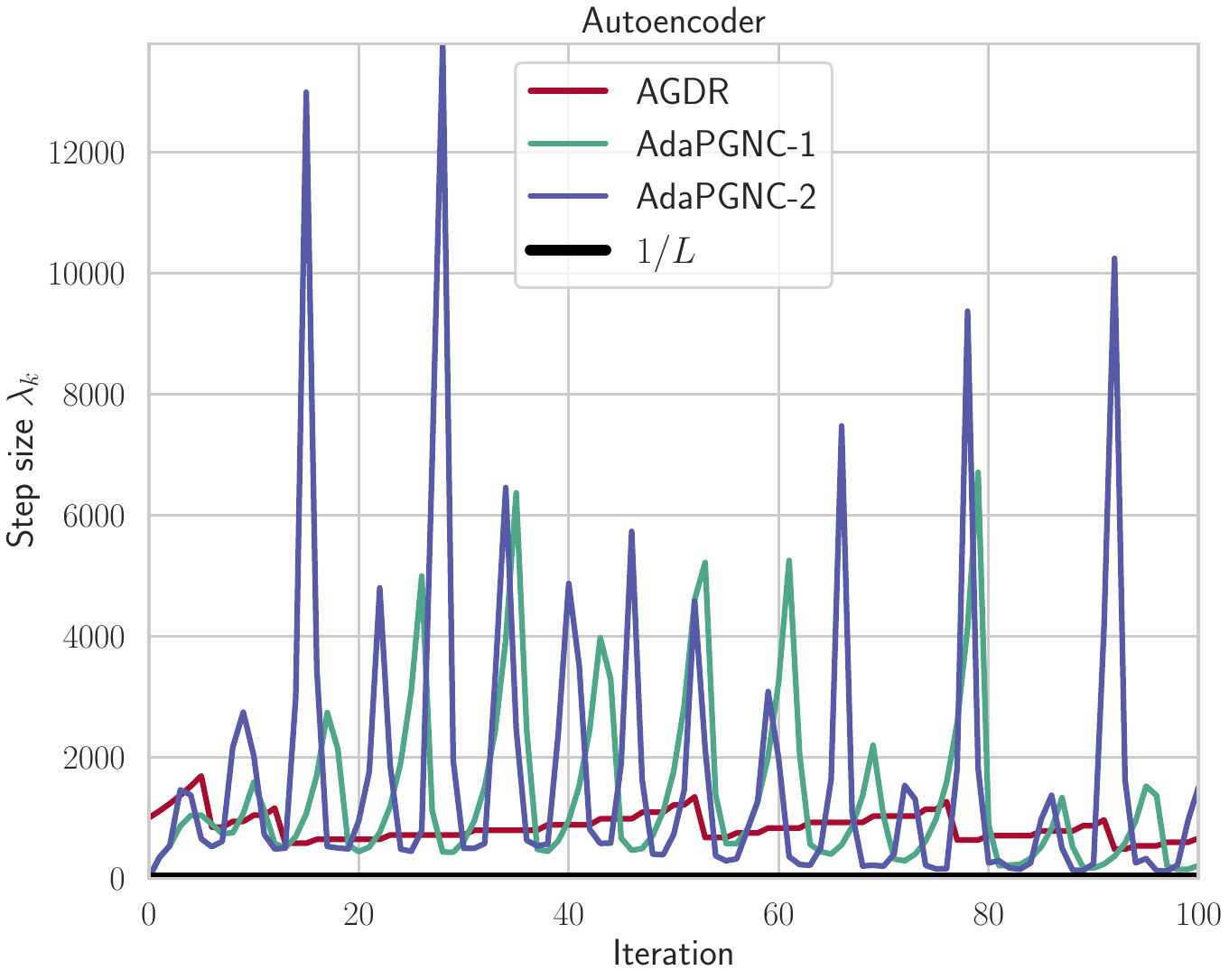}
    \caption{Comparison results for the autoencoder training problem.}
    \label{fig:autoencoder-result}
\end{figure}

As shown in Figure \ref{fig:autoencoder-result}, AdaPGNC-2 achieves the steepest reduction in function values, followed by AdaPGNC-1, AGDR, and GD-LS. In terms of gradient norms, AGDR reaches its minimum value the fastest, with AdaPGNC-2, AdaPGNC-1, and GD-LS following in sequence. The plateaus observed in the gradient norm plot after approximately 20 seconds occur because each algorithm has found its minimum achievable gradient norm by that time; since we report the best value found up to $T$ seconds, the curves flatten once no further improvement is made. Consistent with previous observations, both AdaPGNC variants attain significantly larger step sizes than AGDR. This likely explains the superior performance of AdaPGNC variants in minimizing the objective function.

\medskip 
\noindent \textbf{Low-rank matrix completion problem.}
We next consider the low-rank matrix completion problem \cite{wen2018survey, yao2018efficient}, formulated as follows:
\begin{align}\label{matrixcompletion}
\min_{\substack{Z= (U,V)\in\mathbb{R}^{p\times r} \times \mathbb{R}^{q\times r}}} f(Z) = 
\frac{1}{2N}\sum_{(i,j,s)\in\Omega}\left((UV^\top)_{ij}-s\right)^2 + \frac{1}{2N}\left\| U^\top U - V^\top V\right\|_F^2.
\end{align}
The goal is to find a rank-$r$ approximation $UV^\top$ that completes missing entries in a matrix, using a set $\Omega$ of $N$ observed entries $(i,j,s)$ where $s$ denotes the value of the $(i,j)$-th element. The regularization term $||U^\top U - V^\top V||_F^2$ ensures consistency between the factors $U \in \mathbb{R}^{p \times r}$ and $V \in \mathbb{R}^{q \times r}$ (see \cite{tu2016low} for details). We evaluated the algorithms on the MovieLens-100K and MovieLens-1M collaborative filtering datasets \cite{harper2015movielens}. Table \ref{tableofmatrixcompletion} presents the details about the data and matrix dimensions for both datasets. We terminate the compared algorithms either after 600 seconds have elapsed or when the gradient norm is found to be less than $10^{-10}$.  The experimental results are shown in Figures~\ref{fig:mc-100k-100-result} and \ref{fig:mc-1m-500-result} for the two datasets.

\begin{table}[htb!]
    \centering
    \caption{Details about the data and matrix dimensions for the matrix completion problem \eqref{matrixcompletion}.}
    \begin{tabular}{l|c|c}
        \toprule
        \textbf{Dataset} & \textbf{MovieLens-100K} & \textbf{MovieLens-1M}\\
        \midrule
        \textbf{Rank: }$r$ & $100$ & $500$\\
        \textbf{Row dimension of $U$: }$p$ & $943$ & $6,\!040$ \\
        \textbf{Row dimension of $V$: }$q$ & $1,\!682$ & $3,\!900$\\
        \textbf{Number of observed entries: }$N$ & $100,\!000$ & $1,\!000,\!209$\\
        \textbf{Number of parameters: }$(p+q)r$ & $262,\!500$ & $4,\!970,\!000$\\
        \bottomrule
    \end{tabular}
    \label{tableofmatrixcompletion}
\end{table}

\begin{figure}[htb!]
    \centering

    \includegraphics[width=0.3\textwidth]{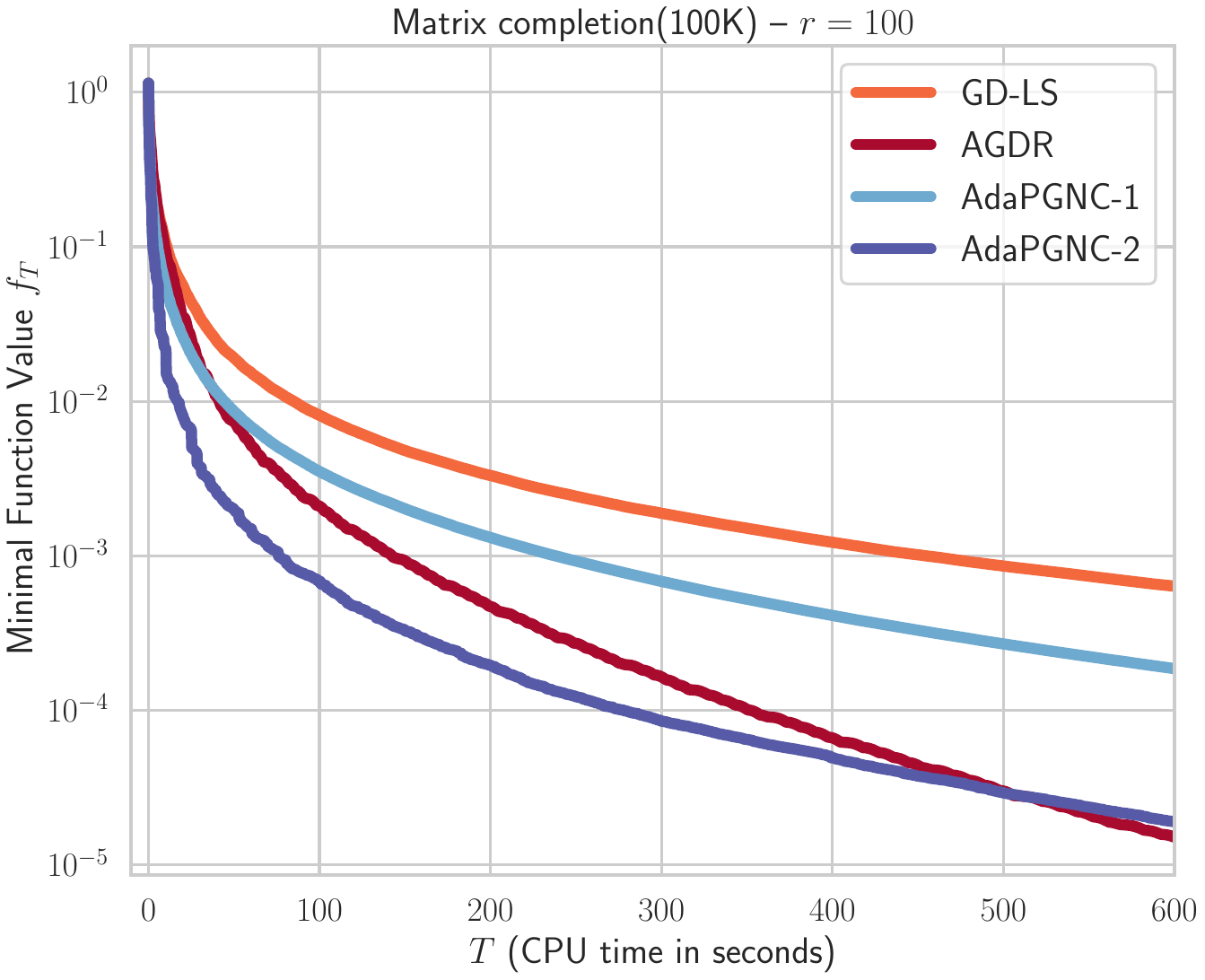}
    \includegraphics[width=0.3\textwidth]{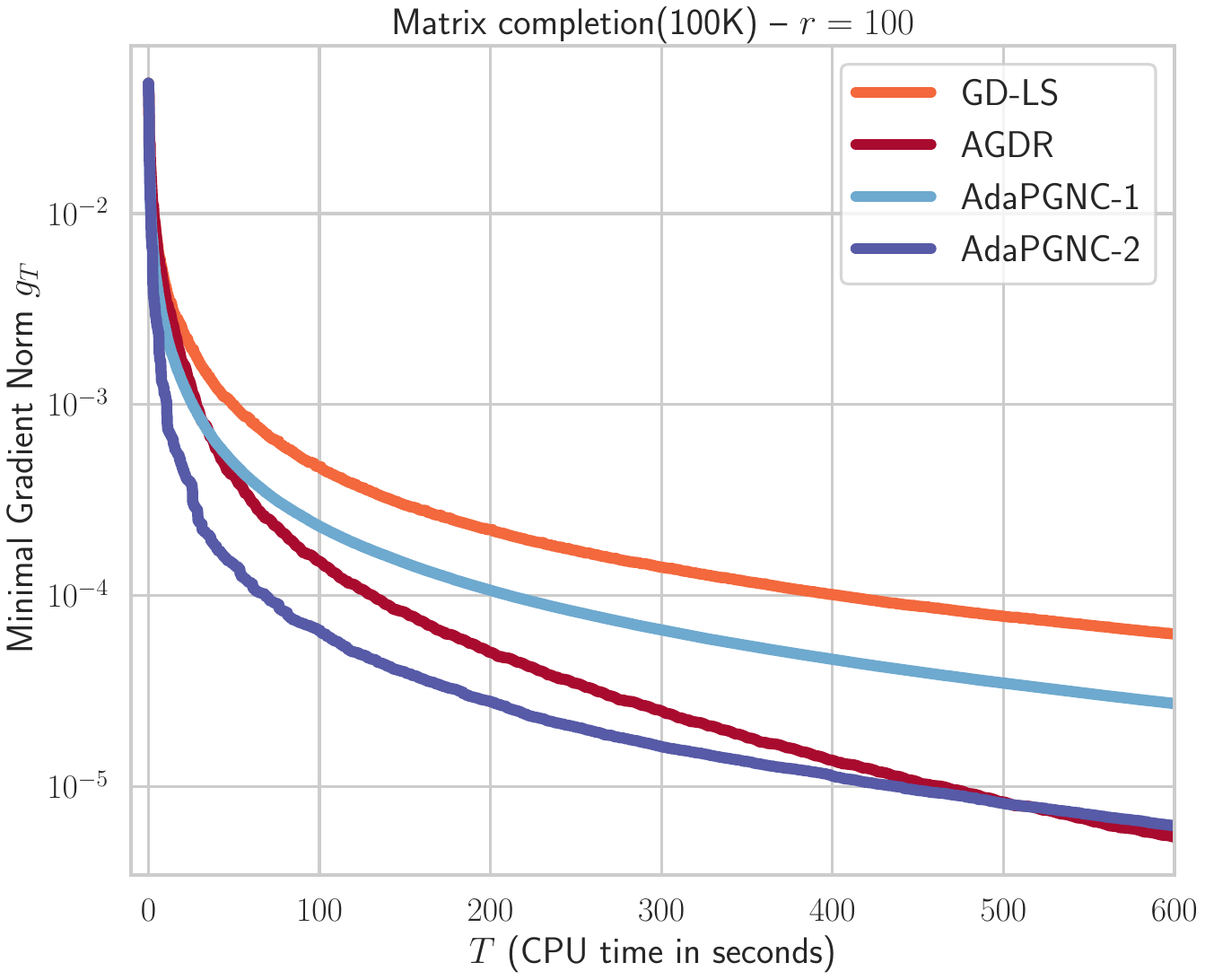}
    \includegraphics[width=0.3\textwidth]{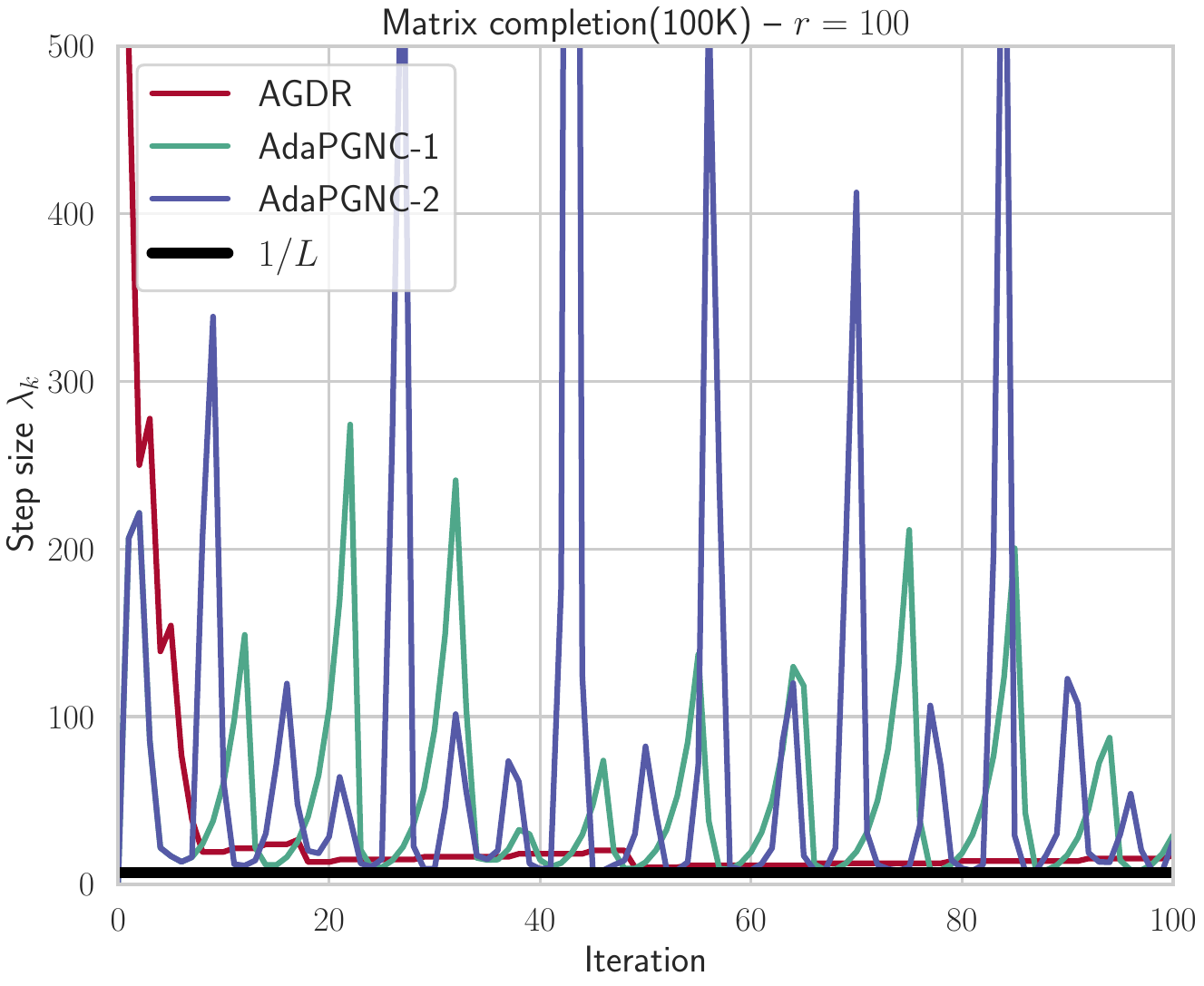}
    \caption{Comparison results for the matrix completion problem on MovieLens-100K.}
    \label{fig:mc-100k-100-result}
\end{figure}

\begin{figure}[ht!]
    \centering
    \includegraphics[width=0.3\textwidth]{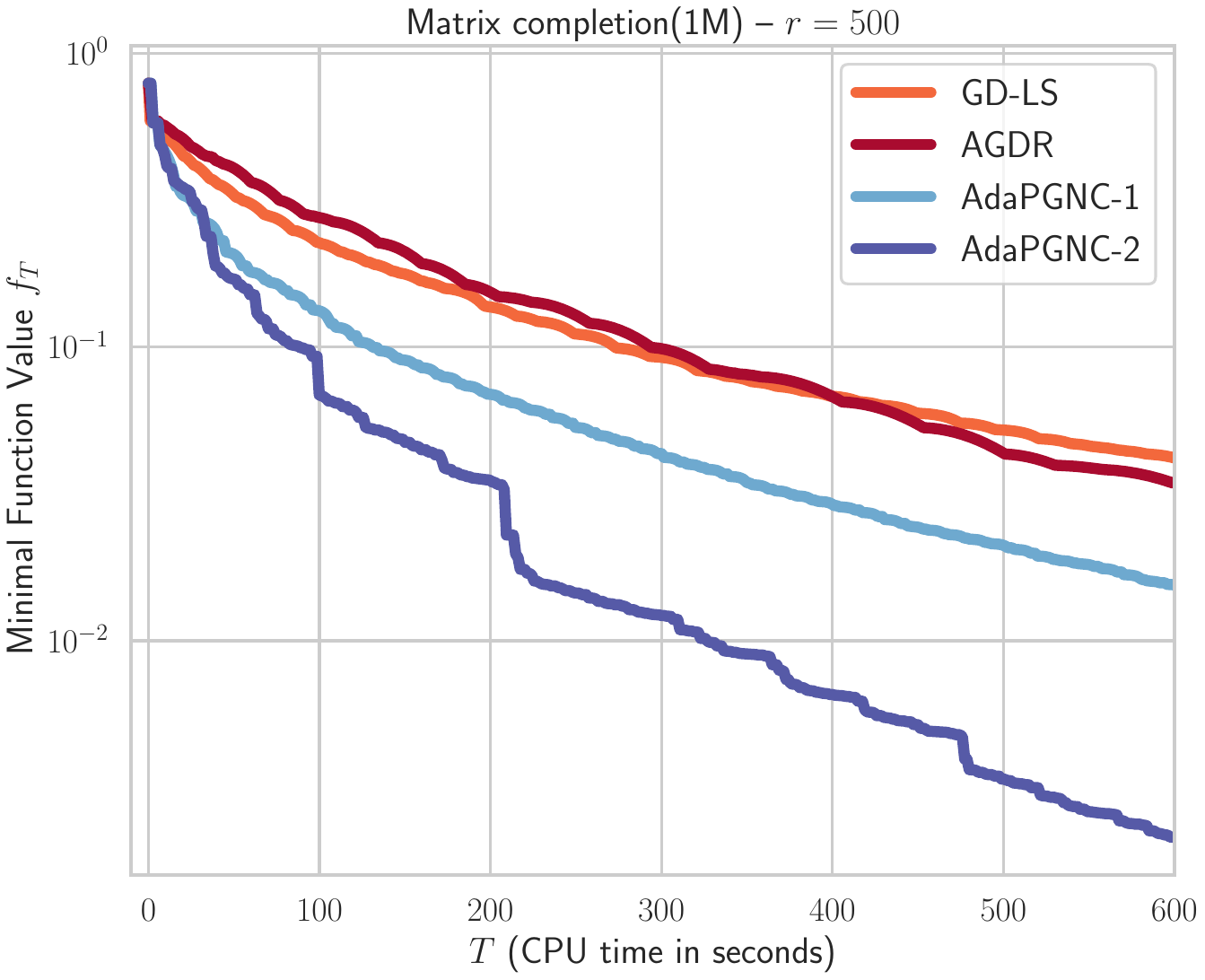}
    \includegraphics[width=0.3\textwidth]{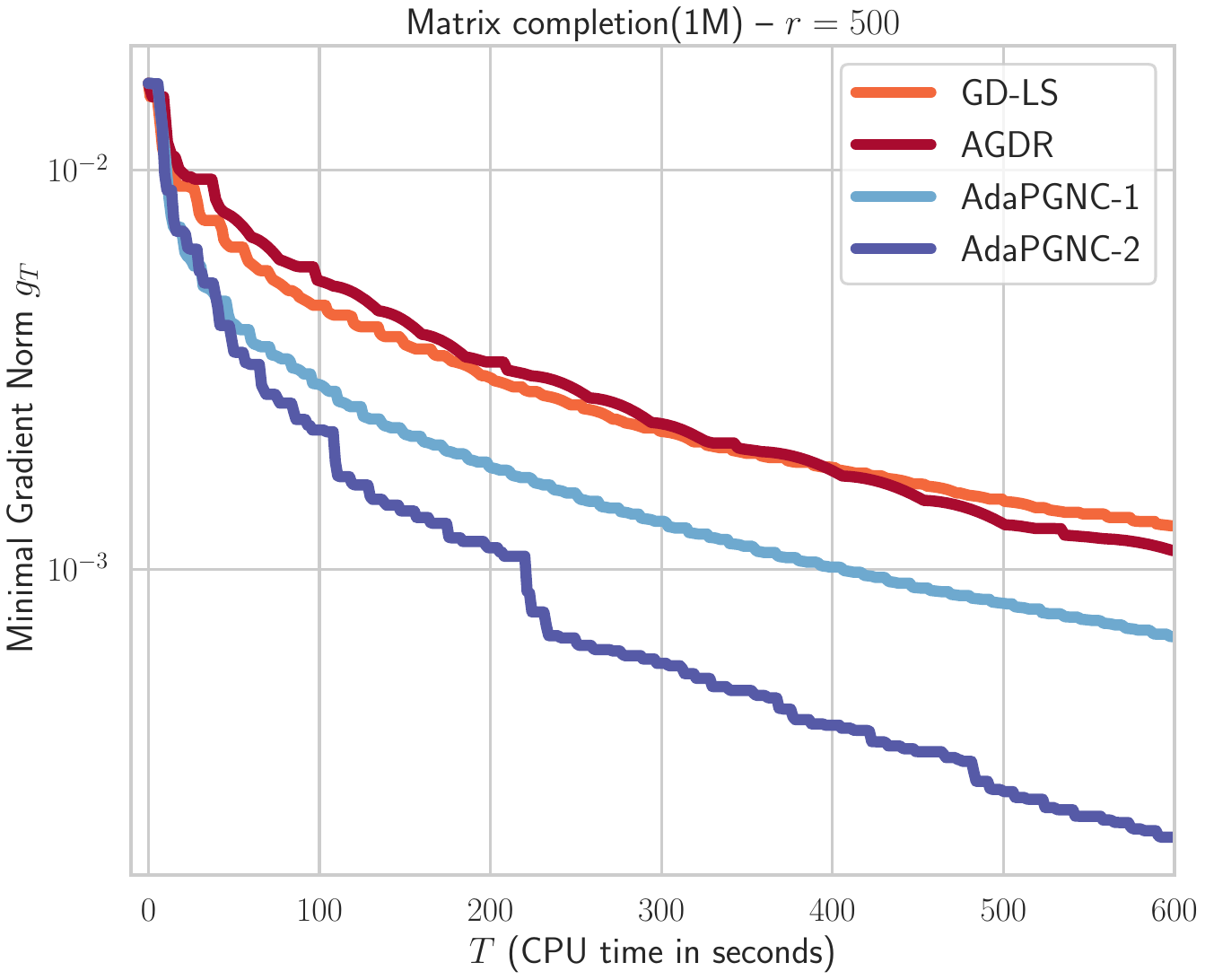}
    \includegraphics[width=0.3\textwidth]{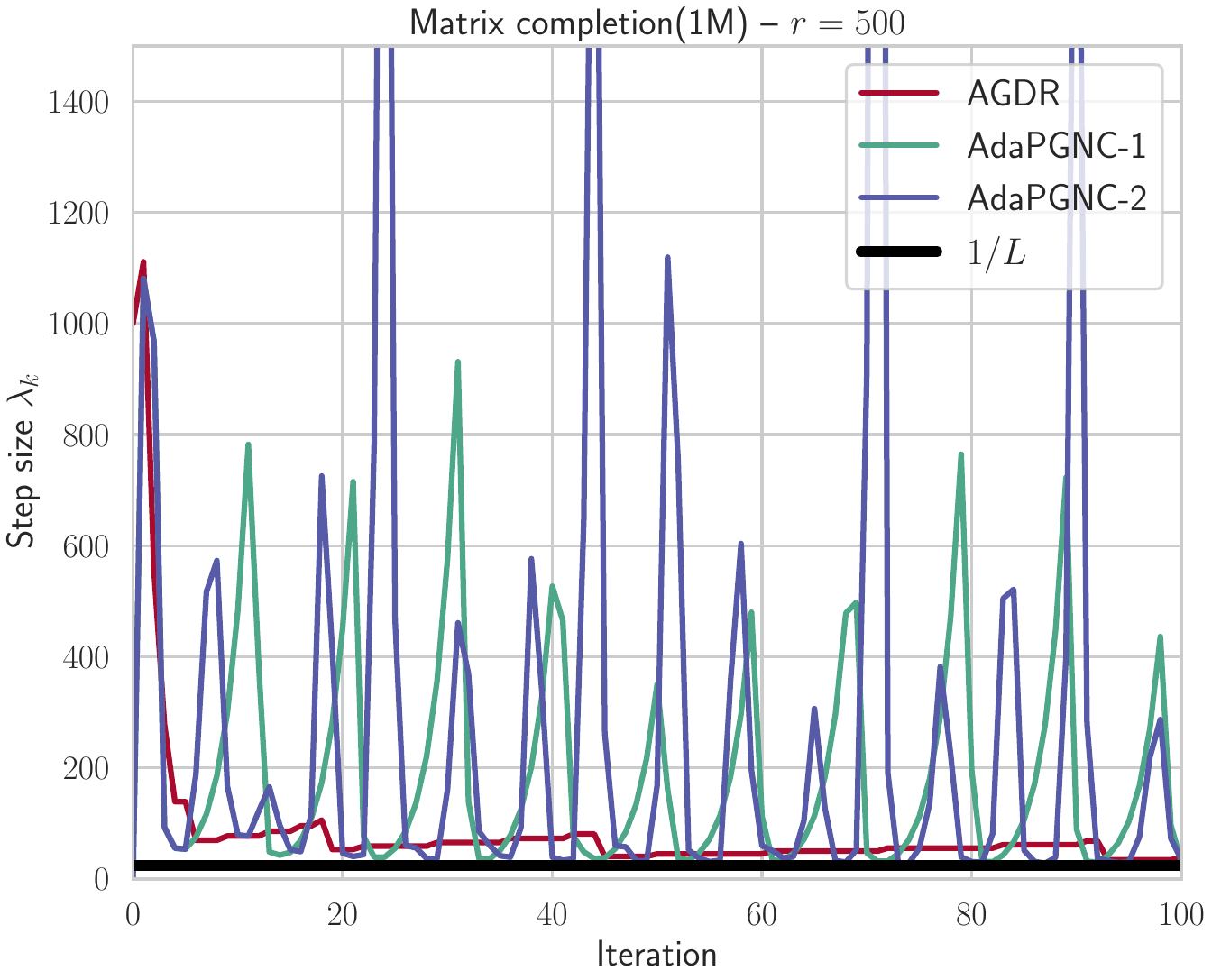}
    \caption{Comparison results for the matrix completion problem on MovieLens-1M.}
    \label{fig:mc-1m-500-result}
\end{figure}

Figures \ref{fig:mc-100k-100-result} and \ref{fig:mc-1m-500-result} show that for the smaller MovieLens-100K dataset, AdaPGNC-2 and AGDR perform best, followed by AdaPGNC-1 and GD-LS. For the larger MovieLens-1M dataset, however, both AdaPGNC-2 and AdaPGNC-1 outperform AGDR and GD-LS significantly in reducing the objective value and gradient norm. This observation suggests that the performance advantage of AdaPGNC algorithms becomes more pronounced as problem dimension increases. Furthermore, the step size plots reveal that AdaPGNC-2 attains significantly larger step sizes than both AdaPGNC-1 and AGDR across both datasets. This consistent behavior demonstrates that AdaPGNC-2's adaptive mechanism adapts more effectively to local curvature, enabling more efficient progress toward convergence.

\medskip 
\noindent \textbf{Nonnegative matrix factorization (NMF) problem.} We next consider the following NMF problem \cite{lee1999learning}:
\begin{align}\label{nmf}
\min_{Z:=(U,V)\in\mathbb{R}^{n\times r}_+ \times \mathbb{R}^{m\times r}_+} f(Z) = \frac{1}{2}\|UV^\top-A\|_F^2, 
\end{align}
where $A\in\mathbb{R}^{n\times m}$ is a given data matrix, $r$ is a positive integer representing the factorization rank. This problem does not satisfy Assumption \ref{assumption-composite} because the objective function is not globally $L$-smooth; therefore, our algorithms are treated as heuristic methods for solving it. Consistent with \cite{malitsky2024adaptive}, we generate $A$ by multiplying matrices $B$ and $C^\top$, where $B\in\mathbb{R}^{m\times r}_+$ and $C\in\mathbb{R}^{n\times r}_+$ have entries sampled from a normal distribution $\mathcal{N}(0,1)$. Negative entries of $B$ and $C$ are replaced with zeros. All methods are initialized with two random matrices $(U_0,V_0)$. For each tuple $(n,r,m)$, we run the algorithms with 10 different random seeds. We terminate the compared algorithms when the gradient residual $\|\CG_k\|$ is found to be less than $10^{-6}$. Table \ref{tab:nmf} presents the average comparison results for AdaPGNC and NPG.

\begin{table}[htbp]
\centering
\caption{Average results for the NMF problem \eqref{nmf}.  In the table, ``GradRes'' denotes $\min_k \|\CG_k\|$, where $\|\CG_k\|= \|\CG_{\lambda_k}(Z_k)\|=\|(Z_{k+1}-Z_{k})/\lambda_k\|$;   ``OptGap'' denotes $\min_k f(Z_k) - f_*$, with $f_*$ approximated by the minimum of $f(Z_k)$ over all iterations and tested algorithms; and Time is measured in seconds.  Best results are highlighted in bold.}
\begin{tabular}{@{}cclcccc@{}}
\toprule
\multicolumn{2}{c}{\textbf{Dimensions}} & \textbf{Metrics} & \multicolumn{4}{c}{\textbf{Average results for ten random data}} \\
\cmidrule(lr){1-2} \cmidrule(lr){4-7}
\multicolumn{2}{c}{\((n, r, m)\)} & & \textbf{NPG1} & \textbf{NPG2} & \textbf{AdaPGNC-1} & \textbf{AdaPGNC-2} \\
\midrule
\multirow{3}{*}{(2000, 20, 3000)} & & Iterations & 1,368.1 & 1,149.4 & 743.8 & \textbf{651.8} \\
                                     & & GradRes & 9.8e-07 & 9.7e-07 & \textbf{9.4E-07} & \textbf{9.4E-07} \\
                                     & & OptGap  & 1.1e-15 & 6.8e-16 & \textbf{4.5E-16} & 5.8E-16 \\
                                     & & Time & 76.3 & 63.8 & 42.3 & \textbf{36.6} \\
\midrule
\multirow{3}{*}{(3000, 20, 2000)} & & Iterations & 1,411.8 & 1,167 & 752.6 & \textbf{645.3} \\
                                    & & GradRes & 9.8e-07 & 9.8e-07 & 9.6E-07 & \textbf{9.4E-07} \\
                                    & & OptGap  & 5.8e-16 & 8.1e-16 & 3.1E-16 & \textbf{2.1E-16} \\
                                    & & Time & 83.7 & 69.0 & 44.8 & \textbf{38.4} \\
\midrule
\multirow{3}{*}{(3000, 20, 3000)} & & Iterations & 1,496.7 & 1,233.9 & 758.1 & \textbf{644.4} \\
                                     & & GradRes & 9.9e-07 & 9.8e-07 & 9.6E-07 & \textbf{9.1E-07} \\
                                     & & OptGap  & 1.3e-15 & 9.5e-16 & 6.7E-16 & \textbf{4.4E-17} \\
                                     & & Time & 123.1 & 101.7 & 62.8 & \textbf{53.5} \\
\midrule
\multirow{3}{*}{(2000, 30, 3000)} & & Iterations & 2,864.4 & 2,424.2 & 1,360.2 & \textbf{1,178.6} \\
                                     & & GradRes & 9.9E-07 & 9.9E-07 & 9.7E-07 & \textbf{9.5E-07} \\
                                     & & OptGap  & 6.9E-16 & 1.1E-15 & 7.1E-16 & \textbf{2.4E-16} \\
                                     & & Time & 167.4 & 141.9 & 81.0 & \textbf{70.4} \\
\midrule
\multirow{3}{*}{(3000, 30, 2000)} & & Iterations & 2,875.3 & 2,461.5 & 1,391.2 & \textbf{1,178.6} \\
                                     & & GradRes & 9.9E-07 & 9.9E-07 & 9.8E-07 & \textbf{9.7E-07} \\
                                     & & OptGap  & 1.1E-15 & 9E-16 & 5.3E-16 & \textbf{7.8E-17} \\
                                     & & Time & 178.0 & 152.4 & 82.3 & \textbf{75.0} \\
\midrule
\multirow{3}{*}{(3000, 30, 3000)} & & Iterations & 3,076.7 & 2,582 & 1,362.6 & \textbf{1,130.4} \\
                                     & & GradRes & 9.9E-07 & 9.9E-07 & 9.8E-07 & \textbf{9.6E-07} \\
                                     & & OptGap  & 7.1E-16 & 9.2E-16 & 4.9E-16 & \textbf{2.2E-16} \\
                                     & & Time & 261.6 & 219.2 & 117.3 & \textbf{97.3} \\
\bottomrule
\end{tabular}

\label{tab:nmf}
\end{table}

Table \ref{tab:nmf} shows that across all tested scenarios, AdaPGNC-2 almost always achieves the best performance in terms of number of iterations, CPU time, and optimality residuals (GradRes and OptGap in the table). AdaPGNC-1 performs slightly worse than AdaPGNC-2, and both AdaPGNC variants outperform the two NPG variants significantly.

\subsection{Experiments on a convex logistic regression problem}
In this subsection, we consider the following convex logistic regression problem with $\ell_2$-square regularization:
\begin{align}
\min_{x\in\mathbb{R}^n} f(x) = -\frac{1}{m}\sum_{i=1}^{m}\left[y_i\log(\sigma(a_i^\top x))+(1-y_i)\log(1-\sigma(a_i^\top x))\right]+\frac{\gamma}{2}\|x\|^2,\label{logistic-regression}
\end{align}
where $a_i\in\mathbb{R}^n$, $y_i\in\{0,1\}$, $\sigma(z)= 1/(1+\exp(-z))$ is the sigmoid function, $m$ is the number of observations, and $\gamma > 0$ is a regularization parameter. The gradient of $f$ is given by $\nabla f(x) = \frac{1}{m} \sum_{i=1}^m a_i (\sigma(a_i^\top x) - y_i) + \gamma x$, and $f$ is $L$-smooth with $L = \frac{1}{4} \lambda_{\max}(A^\top A) + \gamma$, where $A = (a_1, \dots, a_m)^\top$ and $\lambda_{\max}(A^\top A)$ stands for the largest eigenvalue of $AA^\top$. We conducted experiments on the mushrooms, w8a, and covtype datasets from LIBSVM \cite{chang2011libsvm}. Details and parameters of these datasets are provided in Table \ref{tableofdatasetoflog}.

\begin{table}[ht!]
    \centering
    \renewcommand{\arraystretch}{1.2}
    \caption{Datasets used in the logistic regression problem.}
    \begin{tabular}{c|ccccc}
        \toprule
        \textbf{Dataset} & $m$ & $n$ & $L$ & $\gamma$ & Max Iteration\\
        \midrule
        \textbf{mushrooms} & $8124$ & $112$ & $2.59$ & $L/m$ & $1,\!000$\\

        \textbf{w8a} & $49,\!749$ & $300$ & $0.66$ & $L/m$ & $3,\!000$\\

        \textbf{covtype} & $581,\!012$ & $54$ & $5.04\times10^6$ & $L/(10m)$ & $10,\!000$ \\
        \bottomrule
    \end{tabular}
    
    \label{tableofdatasetoflog}
\end{table}

For this test, we compared AdaPGNC-BB-1 and AdaPGNC-BB-2 with AdaBB and AdaBB3 (see Table \ref{tab:algorithm_config}). All four algorithms use an initial step size of $\lambda_0 = 10^{-10}$. Figure~\ref{fig:logistic-regression} plots the optimality gap $\min_{0\leq i\leq k} (f(x_i) - f_*)$ and gradient norm $\min_{0\leq i\leq k} \|\nabla f(x_i)\|$ against iteration count, where $f_*$ is approximated by running AdaPGNC-BB-1 and AdaPGNC-BB-2 for an additional 1000 iterations beyond the specified maximum. The results show that both AdaPGNC variants consistently outperform AdaBB and AdaBB3 across all datasets, demonstrating superior efficiency and effectiveness in solving these logistic regression problems. Additionally, Figure \ref{fig:logistic-regression-stepsize} presents the step sizes generated by the four algorithms, revealing that our methods generally produce larger step sizes than AdaBB and AdaBB3---this helps explain their improved performance.

\begin{figure}[ht!]
    \centering
    \includegraphics[width=0.3\textwidth]{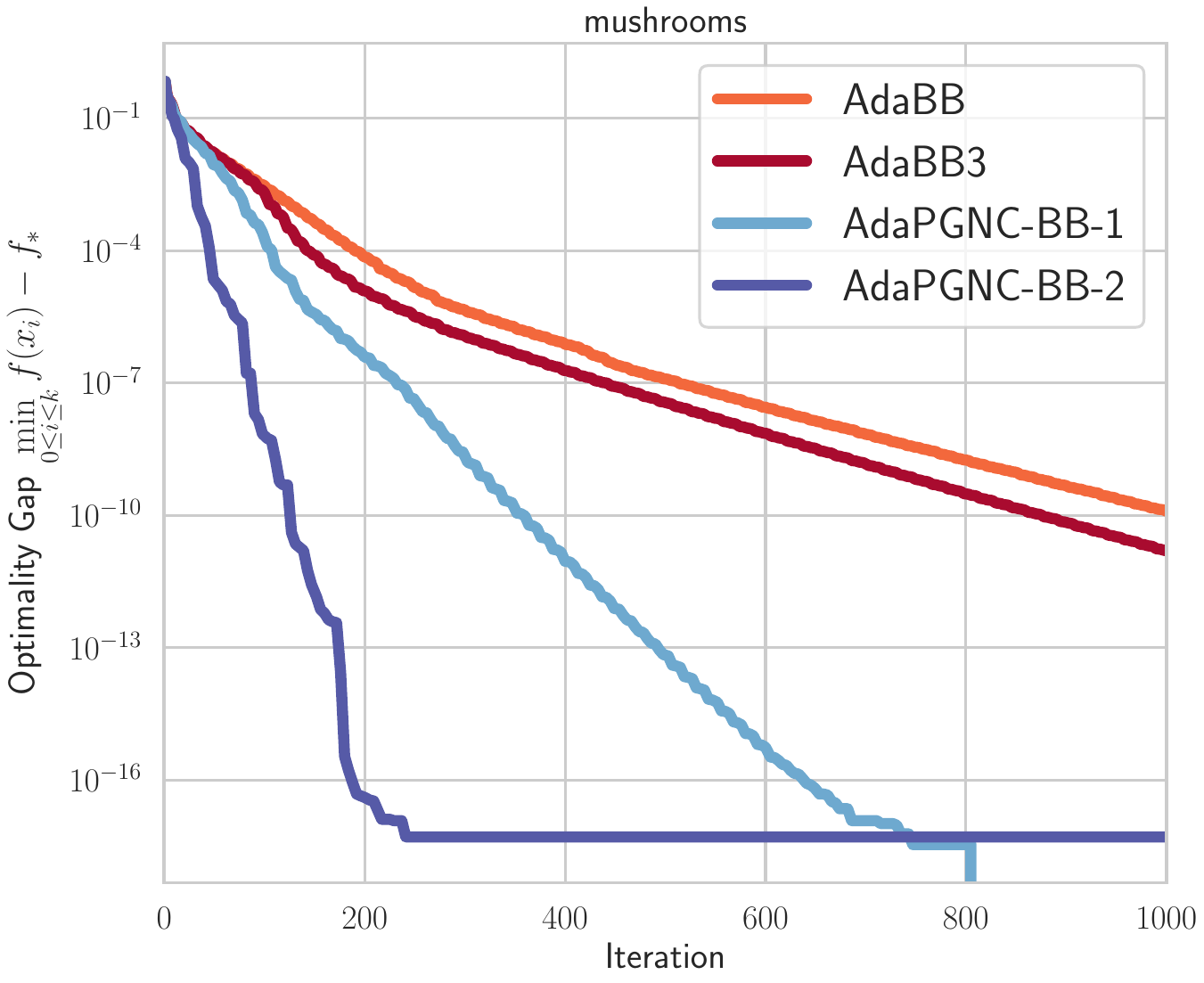}
    \includegraphics[width=0.3\textwidth]{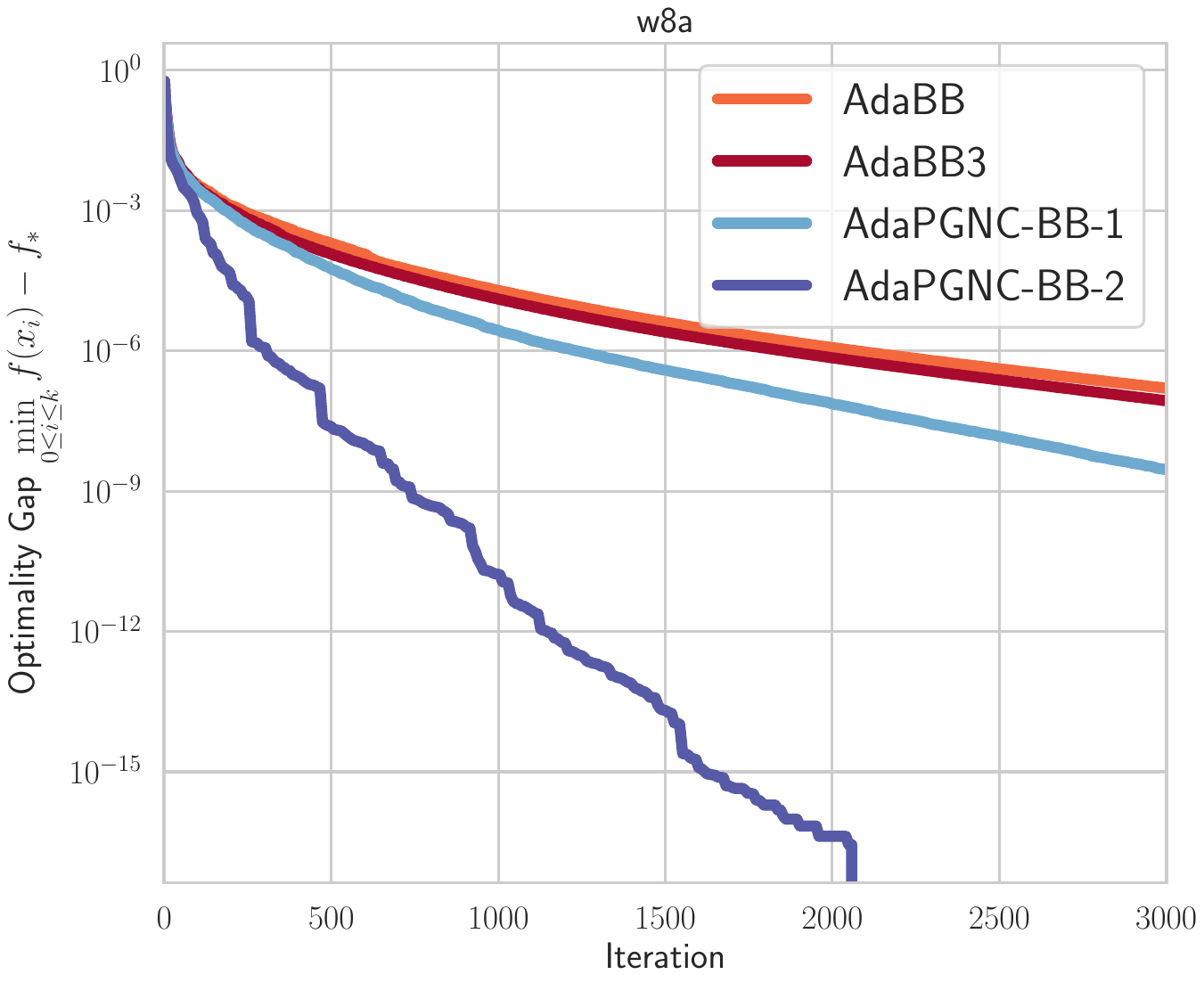}
    \includegraphics[width=0.3\textwidth]{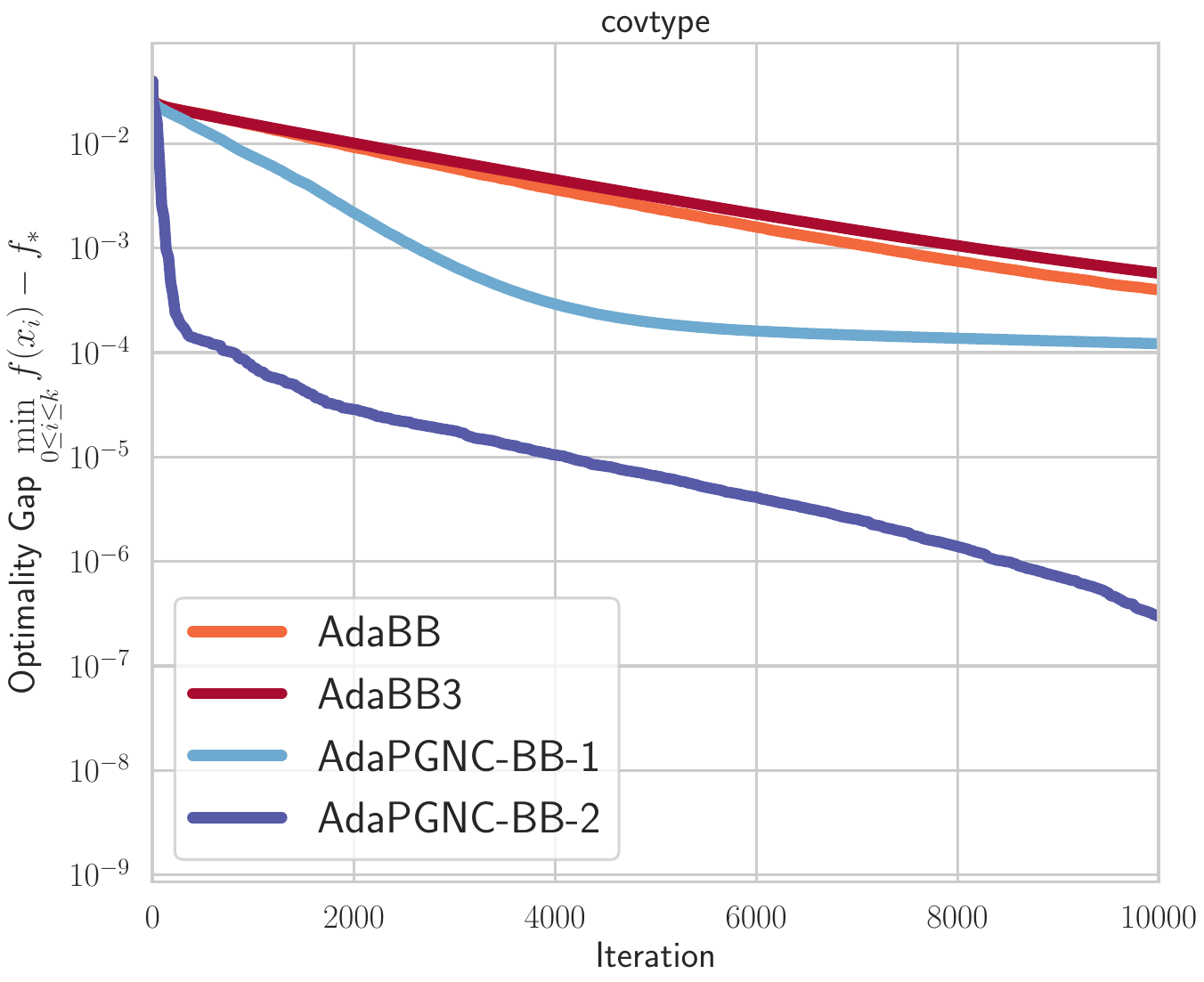}\\

    \includegraphics[width=0.3\textwidth]{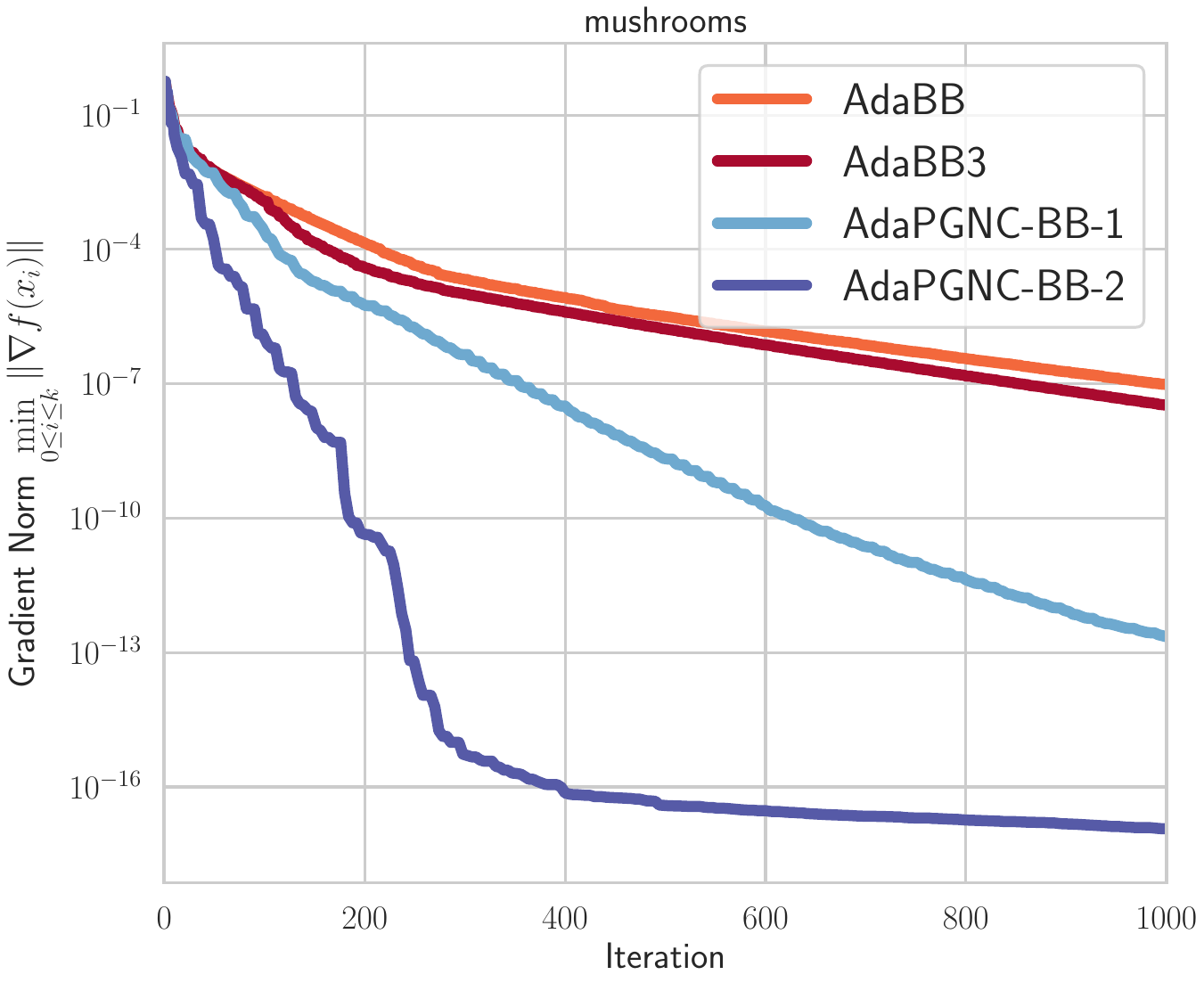}
    \includegraphics[width=0.3\textwidth]{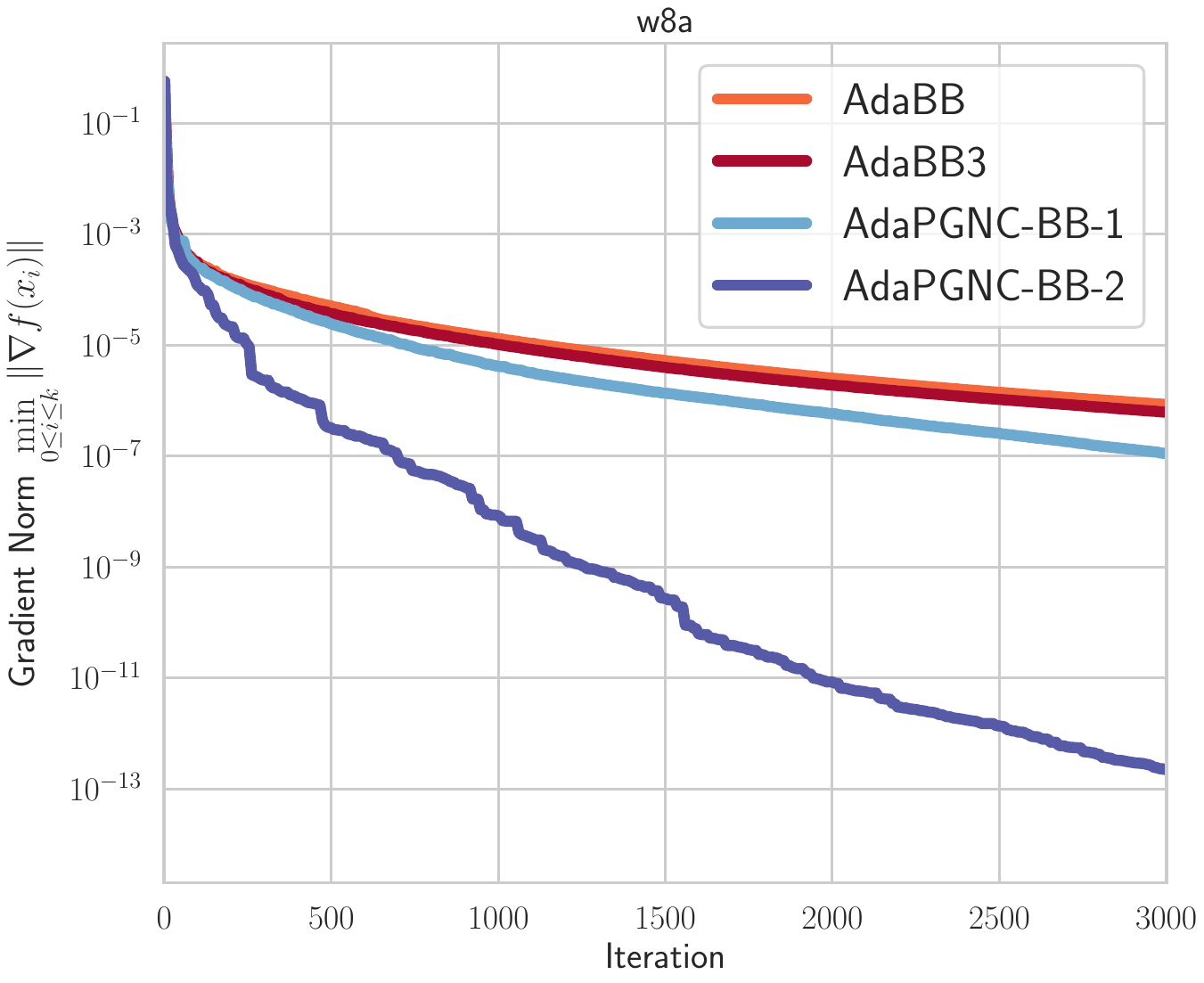}
    \includegraphics[width=0.3\textwidth]{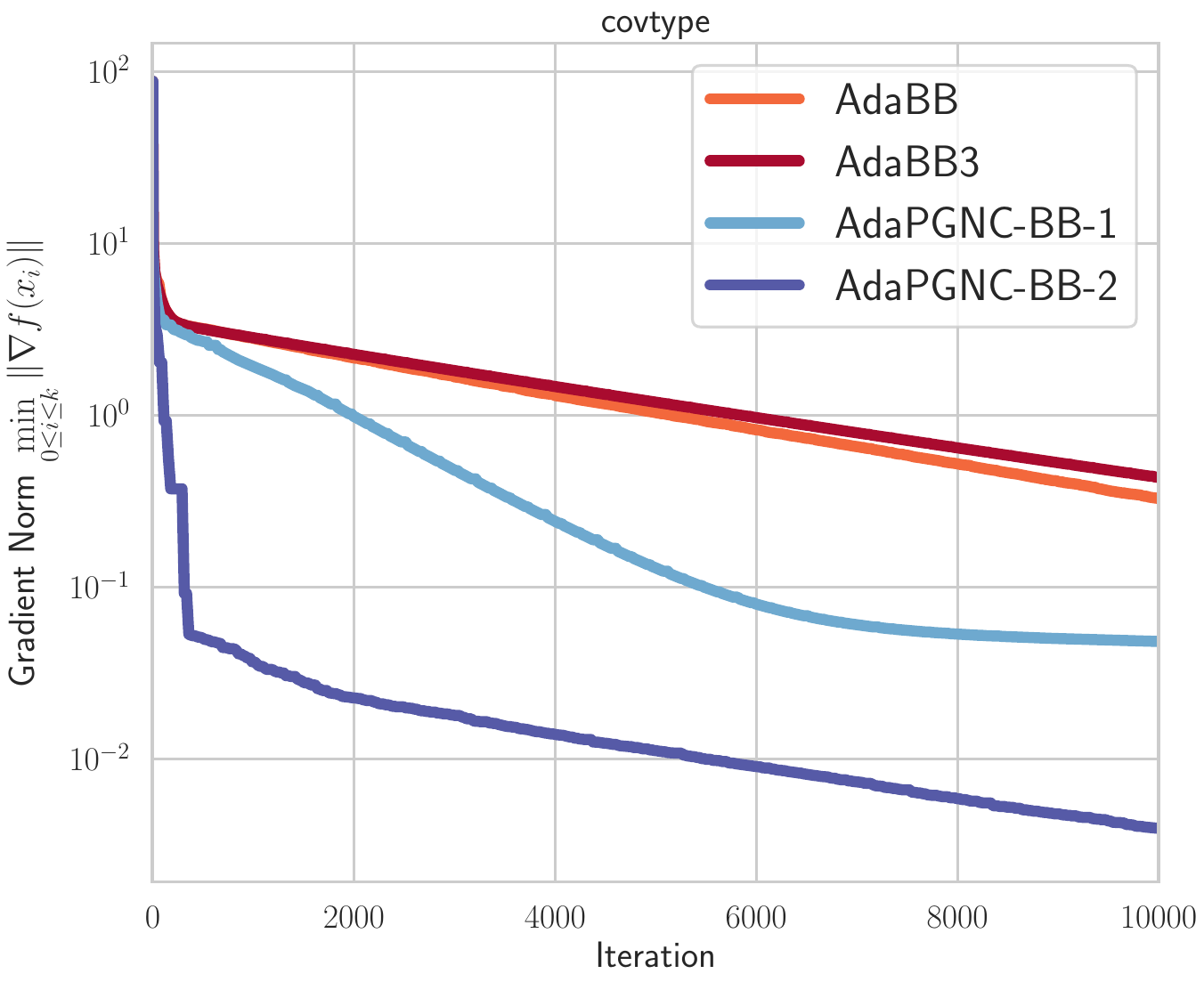}
    
    \caption{Comparison results for the logistic regression problem across the three datasets.}
    \label{fig:logistic-regression}
\end{figure}

\begin{figure}[ht!]
    \centering
    \includegraphics[width=0.3\textwidth]{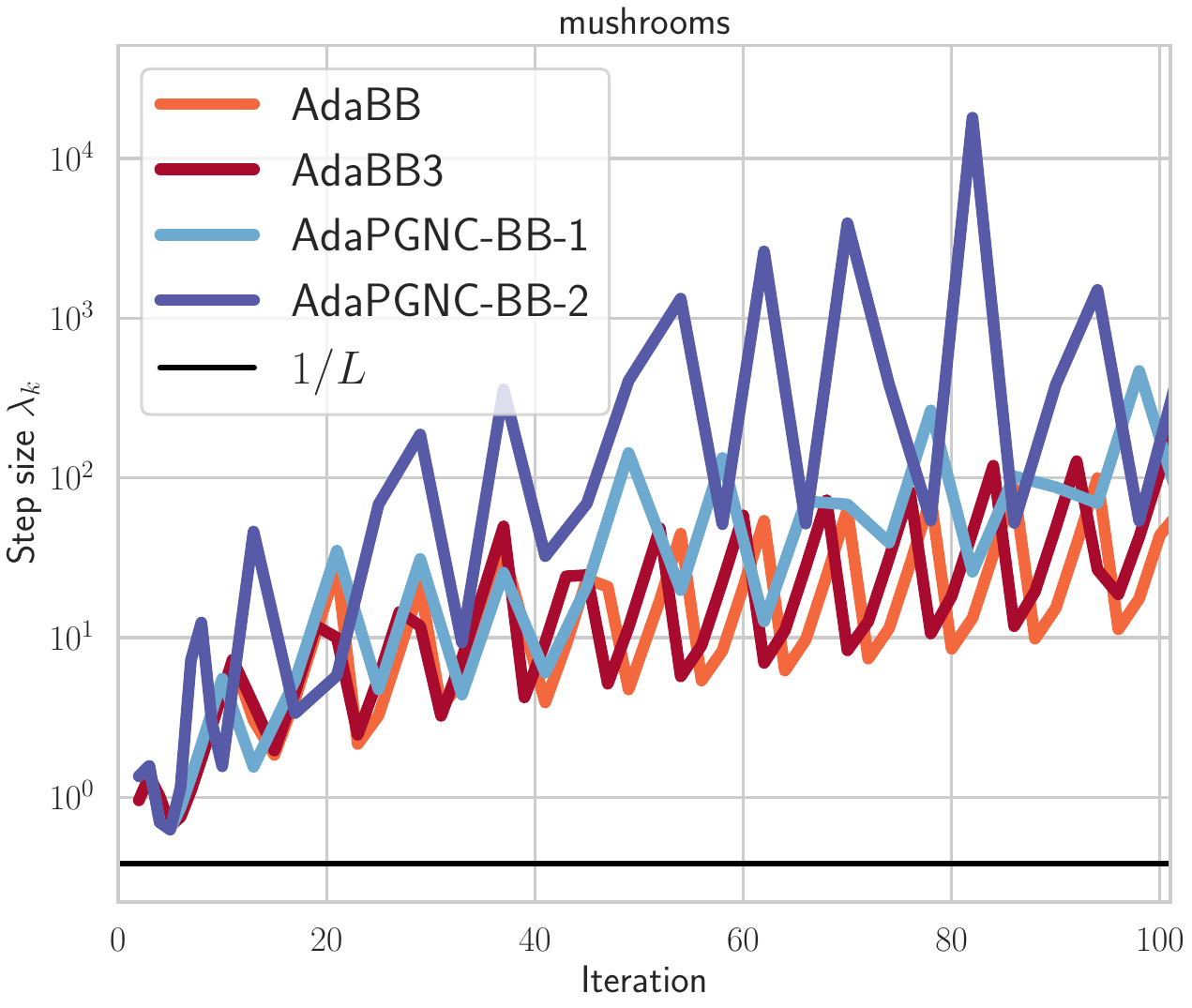}
    \includegraphics[width=0.3\textwidth]{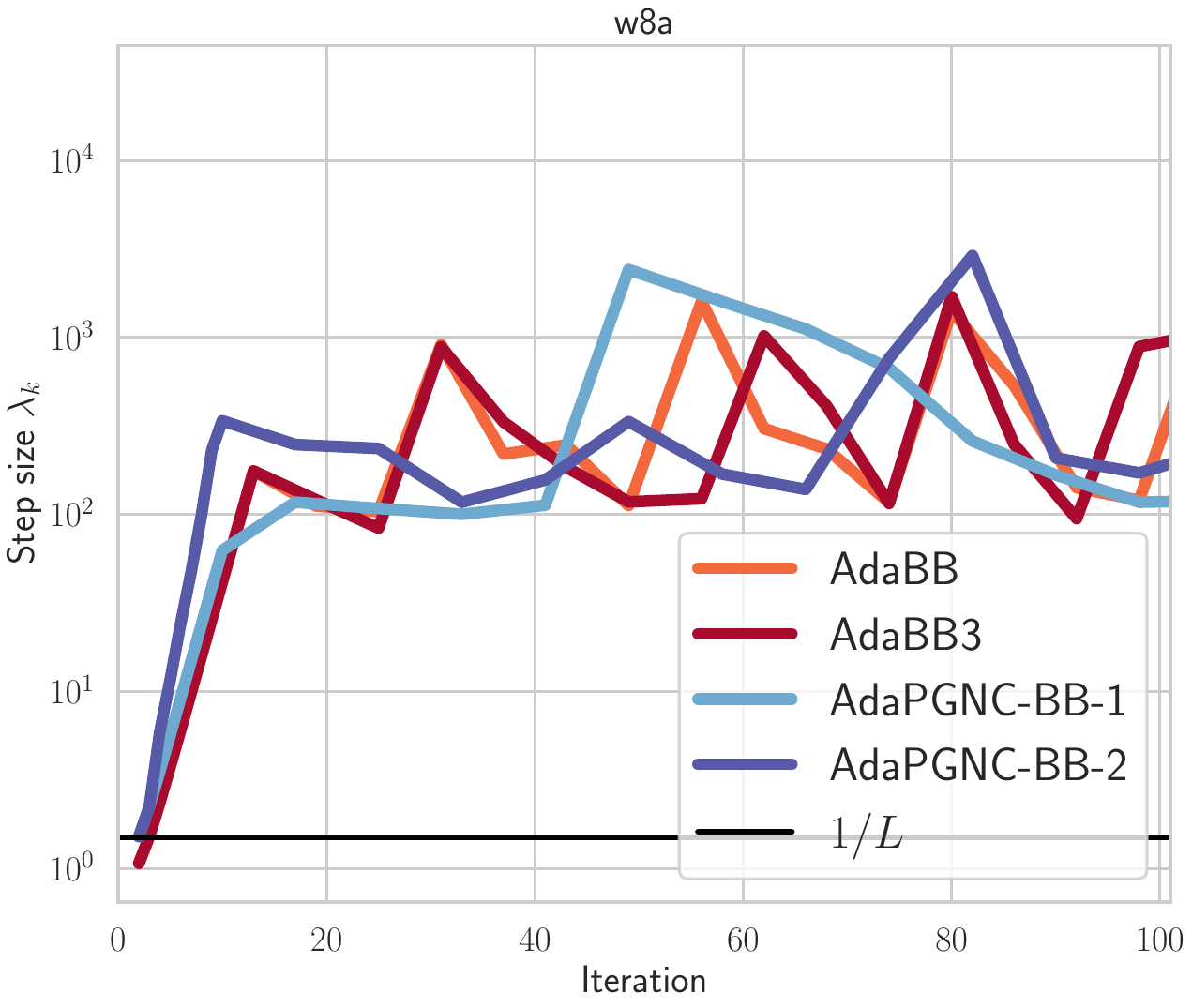}
    \includegraphics[width=0.3\textwidth]{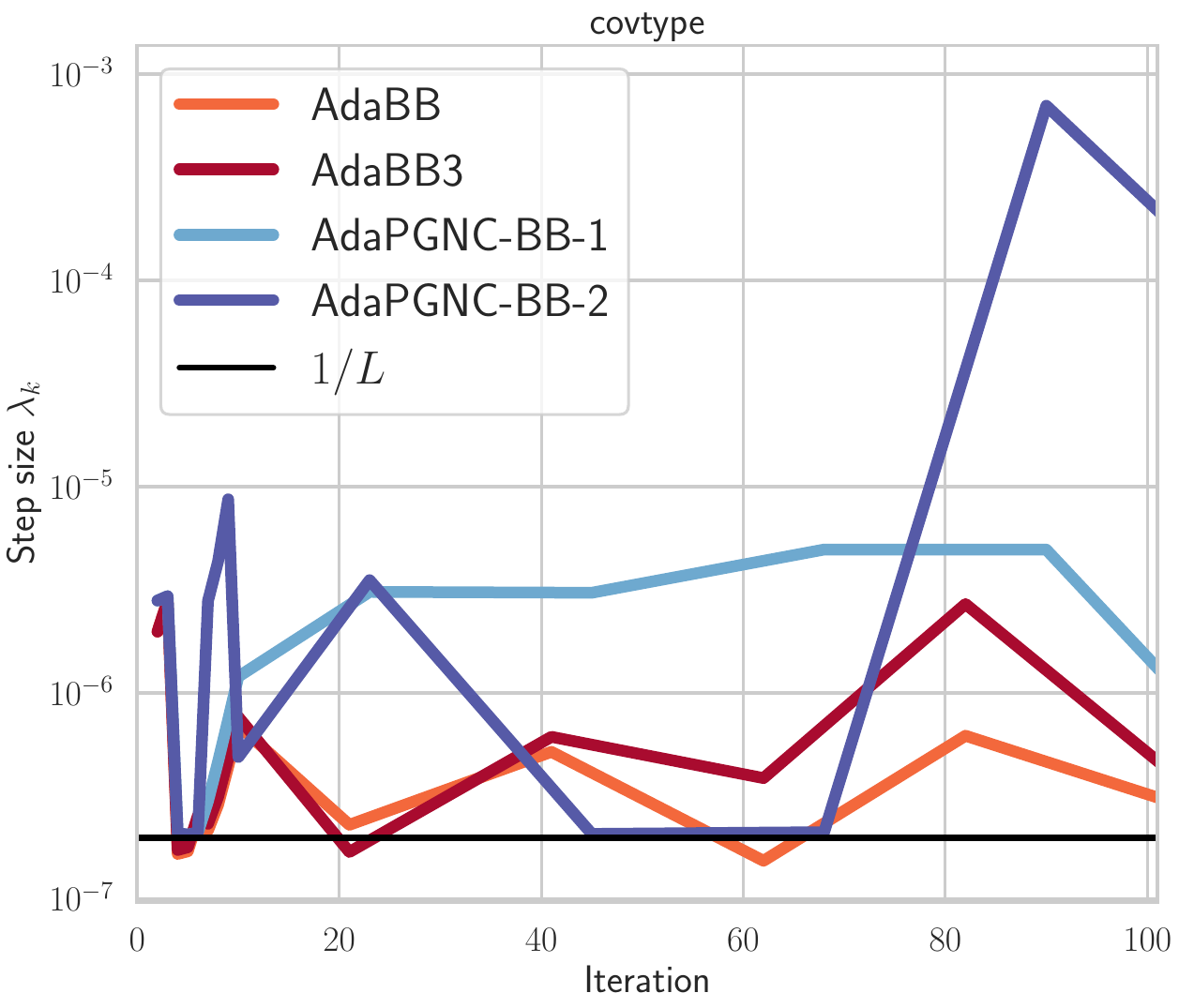}
    
    \caption{Step sizes generated by AdaBB, AdaBB3, AdaPGNC-BB-1 and AdaPGNC-BB-2.}
    \label{fig:logistic-regression-stepsize}
\end{figure}

\section{Conclusions}\label{sectionconclusion}
This paper proposes AdaPGNC, a parameter-free and line-search-free adaptive proximal gradient method for general nonconvex optimization problems. The proposed algorithm extends the capabilities of classical adaptive gradient methods to general composite nonconvex problems, while eliminating the need for the bounded-iterate assumption or second-order smoothness requirements. Furthermore, it features a single-loop structure and requires no restarts. Theoretically, we construct a novel Lyapunov function to establish convergence rates under mild conditions, while simultaneously resolving the challenging question raised in \cite{lan2024optimalparameterfreegradientminimization} regarding the feasibility of simplified parameter-free frameworks for nonconvex optimization. Thus, this work bridges the gap between algorithmic simplicity and theoretical guarantees for general composite nonconvex optimization problems. Even in the convex setting, our analysis offers new insights into the design of adaptive gradient methods—for instance, the short Barzilai-Borwein (BB) step size can be seamlessly incorporated into the step size design. Numerical experiments further demonstrate the competitive performance of AdaPGNC against various benchmark parameter-free methods for both nonconvex and convex optimization problems, validating its practical efficiency alongside its theoretical guarantees. An interesting direction for future work is to develop meaningful heuristics for determining the nonnegative summable sequence $\{\rho_k\}$ dynamically. Currently, this sequence must be determined offline.

\bibliographystyle{siamplain}
\bibliography{ref}
\end{document}